\documentclass{article}
\usepackage[english]{babel}
\usepackage{tikz-cd}
\usepackage{amssymb}
\usepackage{amsthm}
\usepackage{amsmath}
\usepackage{booktabs}
\usepackage{amsfonts}
\usepackage{amsthm}
\usepackage{caption}
\usepackage{graphicx}
\usepackage{xspace}
\usepackage{hyperref}
\usepackage{geometry}
\usepackage{color}
\usepackage{url}
\usepackage[nottoc]{tocbibind}
\usepackage{stmaryrd}
\usepackage{tensor}
\newtheorem{thm}{Theorem}[section]
\newtheorem{lem}[thm]{Lemma}
\newtheorem{prop}[thm]{Proposition}
\newtheorem{cor}[thm]{Corollary}

\theoremstyle{definition}
\newtheorem{dfn}[thm]{Definition}

\newtheorem{ex}[thm]{Example}
\usepackage[all]{xy}

\title{On homotopy exponentials in path categories}
\author{Martijn den Besten}
\begin{document}
\date{}
\maketitle

\begin{abstract}
In the first part of this paper we show that path categories are enriched over groupoids, in a way that is compatible with a suitable 2-category of path categories. In the second part we introduce a new notion of homotopy exponential and homotopy Pi-type for path categories, the strong homotopy exponential and the strong homotopy Pi-type. We prove some of their basic properties, along with two important technical propositions. The results in the second part of this paper are mostly obtained using the groupoid enrichment of path categories established in the first part.
\end{abstract}

\section{Introduction} \label{sec1}
Perhaps the most important development in modern type theory is the discovery of a bridge between homotopy theory and Martin-L\"of type theory, two seemingly disparate subjects. Arguably, this development was set in motion by Hofmann and Streicher in \cite{MR1686862}, where a new light is shed on Martin-L\"of's identity type, which was not well-understood up to that point. Hofmann and Streicher show that the identity types do not only induce an equivalence relation on every type, but can be given the structure of a groupoid as well. In turn, they show that types can be modelled by groupoids, by interpreting the morphisms of the groupoid as elements of the identity type. A decade later, Awodey and Warren show that categories carrying a weak factorization system can be interpreted as models of Martin-L\"of type theory, provided that some extra conditions are satisfied \cite{MR2461866}, while on the other hand, Gambino and Garner show that the classifying category associated to a type theory with identity types carries a weak factorization system \cite{MR2469279}. Weak factorization systems are ubiquitous in homotopy theory. The heart of a Quillen model category, a concept which plays a prominent role in homotopy theory and of which topological spaces and simplicial sets are the central examples, consists of a pair of interlinked weak factorization systems. The discovery of this correspondence is what justifies us to think of types as spaces and of identity types as a path spaces.

Path categories where introduced in \cite{Path} by Van den Berg and Moerdijk, strengthening Brown's notion of a category of fibrant objects \cite{MR341469}. Path categories, like Quillen model categories, provide an environment in which one can do abstract homotopy theory. Unlike Quillen model categories, path categories do not carry a weak factorization system. However, it can be shown that something very close to this is true by relaxing the lifting conditions, requiring the upper triangles to commute only up to fiberwise homotopy (see Theorem \ref{liftthm} for more details). The fact that path categories only satisfy this weaker lifting condition is not completely by accident, since the axioms for path categories are motivated by a different version of type theory. Namely, a type theory where all computation rules are formulated as propositional equalities. Part of this correspondence has been made precise in \cite{Prop}, where it is shown that the syntactic category associated to such a type theory carries the structure of a path category. In the translation between type theory and path categories, the weakened computation rule for identity types corresponds to the weaker version of the lifting property.

To formulate and prove interesting statements in a type theory, it helps to have more than just identity types. One of the very basic extensions is that of the (dependent) function type. In this paper, we investigate their counterparts in path categories. These kinds of objects were already introduced in \cite{Path} in two different flavours, the weak homotopy exponentials and $\Pi$-types, and the (ordinary) homotopy exponentials and $\Pi$-types (Definitions 5.1 and 5.2 in \cite{Path}). We propose a third flavour, the strong homotopy exponentials and $\Pi$-types (Definitions \ref{expdef} and \ref{pidef}). The reason for introducing this new notion is that the strong homotopy exponentials and $\Pi$-types have desirable properties that we have been unable to verify for (ordinary) homotopy exponentials and $\Pi$-types, the main examples being Propositions \ref{mainprop1} and \ref{mainprop2}. These Propositions may not look very compelling on their own, but function as important technical tools. Some examples of their uses are Proposition \ref{explem} below and Theorem 2.23 of \cite{MR4083000}. Propositions \ref{mainprop1} and \ref{mainprop2} are also used to prove that weak homotopy exponentials and $\Pi$-types are stable under gluing (Theorems 5.3.1 and 5.4.1 of \cite{Glue}). Furthermore, strong homotopy exponentials and $\Pi$-types are more natural than (ordinary) homotopy exponentials and $\Pi$-types, when presented in terms of functor properties (as in Definitions \ref{expdef} and \ref{pidef}). This may explain why strong homotopy exponentials and $\Pi$-types are easier to work with than (ordinary) homotopy exponentials and $\Pi$-types. Lastly, Proposition \ref{explem} shows that strong homotopy exponentials are good from a type theoretical point of view, since they are precisely those weak homotopy exponentials that satisfy function extensionality.

The precise content of this paper is as follows. In Section \ref{sec2}, we recall some basic definitions and results on path categories, while in Section \ref{sec3}, we show that path categories are enriched over groupoids, in a suitable way. We introduce several notions of homotopy function spaces in Section \ref{sec4} and prove some of their basic properties. In Section \ref{sec5} we prove our main results on homotopy function spaces. Finally, a few somewhat more technical Lemmas are proven in the appendix.

\paragraph{Acknowledgements}
The author would like to thank Benno van den Berg, Menno de Boer and Peter LeFanu Lumsdaine for helpful discussions on the subject matter. In particular, the definition of strong homotopy exponential was suggested by Benno van den Berg and the proof strategy used in Proposition \ref{mainprop1} is due to Peter LeFanu Lumsdaine.

\section{Path categories} \label{sec2}

In this section, we have collected the basic definitions and most of the results we will need on path categories. A complete treatment, including proofs, can be found in Section 2 of \cite{Path}.

\begin{dfn}
A \textit{path category} is a category $\mathcal{C}$ together with two classes of maps, \textit{fibrations} and \textit{weak equivalences}. Fibrations that are weak equivalences as well are called \textit{acyclic fibrations}. If $X \longrightarrow PX \longrightarrow X \times X$ is a factorization of the diagonal on $X$ by a weak equivalence $r$, followed by a fibration $(s,t)$, then $PX$ is a \textit{path object} for $X$. In path category, the following axioms are satisfied.
\begin{enumerate}
\item{Fibrations are closed under composition.}
\item{The pullback of a fibration along any other map exists and is again a fibration.}
\item{The pullback of an acyclic fibration along any other map is again an acyclic fibration}
\item{Weak equivalences satisfy 2-out-of-6. That is, if $f$, $g$ and $h$ are maps such that $hg$ and $gf$ are weak equivalences, then so are $f$, $g$, $h$ and $hgf$.}
\item{Isomorphisms are acyclic fibrations and every acyclic fibration has a section.}
\item{$\mathcal{C}$ has a terminal object, and the unique map $X \longrightarrow 1$ is always a fibration.}
\item{For every object $X$, there is a path object $PX$.}
\end{enumerate}
\end{dfn}

\begin{dfn}
A functor $F : \mathcal{C} \longrightarrow \mathcal{D}$ between path categories is called \textit{homotopical} if it preserves weak equivalences. A functor $F : \mathcal{C} \longrightarrow \mathcal{D}$ between path categories is called \textit{exact} if it preserves weak equivalences, fibrations, terminal objects and pullbacks along fibrations. 
\end{dfn}

\begin{lem}
Let $\mathcal{C}$ be a path category and let $X$ be an object of $\mathcal{C}$. Then the full subcategory $\mathcal{C}(X)$ on the fibrations $Y \longrightarrow X$ in the slice category $\mathcal{C}/X$ is a path category, with the fibrations and weak equivalences inherited from $\mathcal{C}$. We write $P_{X}Y$ for the path object of $Y \longrightarrow X$ in $\mathcal{C}(X)$.
\end{lem}

\begin{proof}
Definition 2.5 of \cite{Path}.
\end{proof}

\begin{prop}
Let $\mathcal{C}$ be a path category. Then for any morphism $f : X \longrightarrow Y$ in $\mathcal{C}$, the pullback functor $f^{*} : \mathcal{C}(Y) \longrightarrow \mathcal{C}(X)$ is exact. In particular, given a fibration $Z \longrightarrow Y$, we may choose $Z \times_{Y} P_{Y}Z$ as path object $P_{X}(Z \times_{Y} X)$.
\end{prop}

\begin{proof}
Proposition 2.6 of \cite{Path}.
\end{proof}

\begin{prop}
The pullback of a weak equivalence along a fibration is again a weak equivalence.
\end{prop}

\begin{proof}
Proposition 2.7 of \cite{Path}.
\end{proof}

\begin{dfn}
Two parallel arrows $f, g : X \longrightarrow Y$ in a path category $\mathcal{C}$ are called \textit{homotopic} if there is a path object $PY$ and a map $H : X \longrightarrow PY$ such that $sH = f$ and $tH =g$. We write $f \simeq g$ or $H : f \simeq g$ in this case.
\end{dfn}

\begin{thm}
The homotopy relation defines a congruence relation on a path category $\mathcal{C}$. We write $\mathsf{Ho}(\mathcal{C})$ for the quotient of $\mathcal{C}$ by this relation.
\end{thm}

\begin{proof}
Theorem 2.14 of \cite{Path}.
\end{proof}

\begin{dfn}
A map $f : X \longrightarrow Y$ in a path category $\mathcal{C}$ is called a \textit{homotopy equivalence} if there exists a map $g : Y \longrightarrow X$ such that $fg$ and $gf$ are homotopic to the identities on $Y$ and $X$.
\end{dfn}

\begin{thm}
Weak equivalences and homotopy equivalences coincide.
\end{thm}

\begin{proof}
Theorem 2.16 of \cite{Path}.
\end{proof}

\begin{dfn}
Let $p : Y \longrightarrow I$  be a fibration in a path category $\mathcal{C}$ and let $f, g : X \longrightarrow Y$ be parallel arrows, such that $pf = pg$. We say that $f$ and $g$ are \textit{fiberwise homotopic} if there is a path object $P_{I}Y$ and a map $H : X \longrightarrow P_{I}Y$ such that $sH = f$ and $tH =g$. We write $f \simeq_{I} g$ or $H : f \simeq_{I} g$ in this case.
\end{dfn}

\begin{thm} \label{liftthm}
If
\begin{equation*}
\begin{tikzcd}
W \arrow[r, "h"] \arrow[d, swap, "w"] & X \arrow[d, "p"] \\
Y \arrow[r, swap, "k"] & Z
\end{tikzcd}
\end{equation*}
is a commutative square with $w$ a weak equivalence and $p$ a fibration, there is a filler $l : Y \longrightarrow X$, unique up to $\simeq_{Z}$, such that $pl = k$ and $lw \simeq_{Z} h$.
\end{thm}

\begin{proof}
Theorem 2.38 of \cite{Path}.
\end{proof}

\begin{cor} \label{cor1}
If
\begin{equation*}
\begin{tikzcd}
W \arrow[r, "h"] \arrow[d, swap, "f"] & X \arrow[d, "g"] \\
Y \arrow[r, swap, "k"] & Z
\end{tikzcd}
\end{equation*}
is a commutative square with $f$ and $g$ fibrations, then $n \simeq_{Y} m$ implies $hn \simeq_{Z} hm$, for maps $m,n : V \longrightarrow W$.
\end{cor}

\begin{proof}
Composing a homotopy $n \simeq_{Z} m$ with a filler for the square
\begin{equation*}
\begin{tikzcd}
W \arrow[r, "rh"] \arrow[d, swap, "r"] & P_{Z} X \arrow[d, "{(s,t)}"] \\
P_{Y}W \arrow[r, swap, "{(hs,ht)}"] & X \times_{Z} X
\end{tikzcd}
\end{equation*}
results in a homotopy $hn \simeq_{Z} hm$.
\end{proof}

\section{Enrichment over groupoids} \label{sec3}

Let $\mathsf{Path}$ denote the 2-category of (small) path categories, homotopical functors and natural transformations, and let $\mathsf{Gpd-Cat}$ denote the 2-category of (small) groupoid enriched categories, 2-functors and strict transformations. We construct a non-trivial 2-functor $\mathfrak{E} : \mathsf{Path} \longrightarrow \mathsf{Gpd-Cat}$ such that composing $\mathfrak{E}$ with the forgetful 2-functor $\mathsf{Gpd-Cat} \longrightarrow \mathsf{Cat}$ yields the forgetful 2-functor $\mathsf{Path} \longrightarrow \mathsf{Cat}$.

\subsection{Action on 0-cells}

Here, we define the action of $\mathfrak{E}$ on 0-cells. In other words, we show that path categories are enriched over groupoids.

\subsubsection{Groupoid structure on hom-sets}

First we show that the hom-sets of a path category $\mathcal{C}$ can be equiped with the structure of a groupoid. This is similar to what is done in the appendix of \cite{Prop}. 

\paragraph{Arrows in hom-groupoids}

Let $X$ and $Y$ be objects of $\mathcal{C}$. In order to give the set $\mathcal{C}(X,Y)$ the structure of a groupoid, we must define its arrows. To this end, let $PY$ be a path object of $Y$. We define the set of (all) arrows of $\mathcal{C}(X,Y)$ as the quotient $\mathcal{C}(X,PY) / \simeq_{Y \times Y}$. The definition of this set is independent (up to canonical isomorphism) of the choice of path object. For if we are given any other path object $P'Y$ of $Y$, we can find fillers $\phi$ and $\psi$ as indicated in the square below.
\begin{equation*}
\begin{tikzcd}
Y \arrow[r, "r'"] \arrow[d, swap, "r"] & P'Y \arrow[d, "{(s',t')}"] \arrow[dl, dashed, shift left, "\psi"] \\
PY \arrow[r, swap, "{(s,t)}"] \arrow[ru, dashed, shift left, "\phi"] & Y \times Y
\end{tikzcd}
\end{equation*}
Since $\psi \phi r \simeq_{Y \times Y} r$ by Corollary \ref{cor1}, we see that both $\psi \phi$ and $1$ are fillers for the square
\begin{equation*}
\begin{tikzcd}
Y \arrow[r, "r"] \arrow[d, swap, "r"] & PY \arrow[d, "{(s,t)}"] \\
PY \arrow[r, swap, "{(s,t)}"] & Y \times Y
\end{tikzcd}
\end{equation*}
which implies that $\psi \phi \simeq_{Y \times Y} 1$. Similarly of course $\phi \psi \simeq_{Y \times Y} 1$. This induces an isomorphism between $\mathcal{C}(X,PY) / \simeq_{Y \times Y}$ and $\mathcal{C}(X,P'Y) / \simeq_{Y \times Y}$, which is well-defined by Corollary \ref{cor1}. Moreover, it is independent of the choice of $\phi$ and $\psi$, since these are unique up to $\simeq_{Y \times Y}$.

\paragraph{Composition and identity in hom-groupoids}

Let $X$ and $Y$ be an objects of $\mathcal{C}$ and $f : X \longrightarrow Y$ a morphism in $\mathcal{C}$. We define the identity arrow, $1_{f}$, on the object $f$ in the groupoid $\mathcal{C}(X,Y)$ to be (the equivalence class of) the homotopy $rf : X \longrightarrow PY$.

Next, we define composition of arrows in the groupoid $\mathcal{C}(X,Y)$. Let $\tau$ be a filler of the following square.
\begin{equation*}
\begin{tikzcd}
Y \arrow[r, "r"] \arrow[d, swap, "{(r,r)}"] & PY \arrow[d, "{(s,t)}"] \\
PY \tensor[_t]{\times}{_s} PY \arrow[r, swap, "{(s \pi_{1},t \pi_{2})}"] & Y \times Y
\end{tikzcd}
\end{equation*}
It is not difficult to check that the map $(r,r)$ is indeed a weak equivalence. We define the composition of (the equivalence classes of) the two homotopies $(H,K) : X \longrightarrow PY \tensor[_t]{\times}{_s} PY$ to be (the equivalence class of) the homotopy $\tau(H,K) : X \longrightarrow PY$. This is well-defined, for if we are given given $(H,K) : X \longrightarrow PY \tensor[_t]{\times}{_s} PY$ with $\mathcal{H} : H \simeq_{Y \times Y} H'$ and  $\mathcal{K} : K \simeq_{Y \times Y} K'$, in other words, a map $(\mathcal{H}, \mathcal{K}) : X \longrightarrow P_{Y \times Y}(PY) \tensor[_{tt}]{\times}{_{ss}} P_{Y \times Y}(PY)$, then composing this with a filler for the square
\begin{equation*}
\begin{tikzcd}[column sep=huge]
PY \tensor[_{t}]{\times}{_{s}} PY \arrow[r, "r \tau"] \arrow[d, swap, "r \times_{Y} r"] & P_{Y \times Y}(PY) \arrow[d, "{(s,t)}"] \\
P_{Y \times Y}(PY) \tensor[_{tt}]{\times}{_{ss}} P_{Y \times Y}(PY) \arrow[r, swap, shift right, "{(\tau(s \pi_{1}, s \pi_{2}), \tau(t \pi_{1}, t \pi_{2}))}"] & PY \times_{Y \times Y} PY
\end{tikzcd}
\end{equation*}
results in a homotopy $\tau(H,K) \simeq_{Y \times Y} \tau(H',K')$. We point out that the map $r \times_{Y} r$ is a weak equivalence, since it is the product of two weak equivalences in $\mathcal{C}(Y)$. Lastly, note that the definition of composition does not depend on the choice of $\tau$.

\paragraph{Groupoid laws for hom-groupoids}

Let $X$ and $Y$ be an objects of $\mathcal{C}$ and $f : X \longrightarrow Y$ a morphism in $\mathcal{C}$. We show that $1_{f}$, as defined in the previous section acts as the identity. For this, it suffices to prove $\tau(1,rt) \simeq_{Y \times Y} 1$ and $\tau(rs,1) \simeq_{Y \times Y} 1$. We will only show the first equivalence, by verifying that both $1$ and $\tau(1,rt)$ fill the square
\begin{equation*}
\begin{tikzcd}
Y \arrow[r, "r"] \arrow[d, swap, "r"] & PY \arrow[d, "{(s,t)}"] \\
PY \arrow[r, swap, "{(s,t)}"] & Y \times Y
\end{tikzcd}
\end{equation*}
For $1$, this is immediate. We also see that
\begin{align*}
(s,t) \tau (1,rt) & = (s \pi_{1} , t \pi_{2}) (1, rt) \\
& = (s, trt) \\
& = (s,t)
\end{align*}
and
\begin{align*}
\tau (1,rt) r & = \tau (r, rtr) \\
& = \tau (r,r) \\
& \simeq_{Y \times Y} r,
\end{align*}
as needed.

We will now show that composition is associative, by providing two fillers for the square
\begin{equation*}
\begin{tikzcd}
Y \arrow[r, "r"] \arrow[d, swap, "{(r,r,r)}"] & PY \arrow[d, "{(s,t)}"] \\
PY \tensor[_t]{\times}{_s} PY \tensor[_t]{\times}{_s} PY \arrow[r, swap, "{(s \pi_{1},t \pi_{3})}"] & Y \times Y
\end{tikzcd}
\end{equation*}
We leave it to the reader to verify that $(r,r,r)$ is a weak equivalence. We claim that $\tau (\pi_{1}, \tau (\pi_{2}, \pi_{3}))$ fills the square. We calculate
\begin{align*}
(s,t) \tau (\pi_{1}, \tau (\pi_{2}, \pi_{3})) &  = (s \pi_{1}, t \pi_{2} ) (\pi_{1}, \tau (\pi_{2}, \pi_{3})) \\
& = (s \pi_{1}, t \tau (\pi_{2}, \pi_{3})) \\
& = (s \pi_{1}, t \pi_{2} (\pi_{2}, \pi_{3})) \\
& = (s \pi_{1}, t \pi_{3})
\end{align*}
and
\begin{align*}
\tau (\pi_{1}, \tau (\pi_{2}, \pi_{3})) (r,r,r) & = \tau (r, \tau (r,r)) \\
& \simeq_{Y \times Y} \tau (r, r) \\
& \simeq_{Y \times Y} r, 
\end{align*}
using $\tau (r, r) \simeq_{Y \times Y} r$ (twice) and the fact that $\tau$ respects $\simeq_{Y \times Y}$, as shown in the previous section. A similar calculation shows that $\tau ( \tau (\pi_{1}, \pi_{2}) , \pi_{3})$ fills the square as well, which implies that $\tau (\pi_{1}, \tau (\pi_{2}, \pi_{3})) \simeq_{Y \times Y} \tau ( \tau (\pi_{1}, \pi_{2}) , \pi_{3})$, as desired.

Lastly, we show that every arrow in the groupoid $\mathcal{C}(X,Y)$ is invertible. To this end, let $\sigma$ be a filler for the square
\begin{equation*}
\begin{tikzcd}
Y \arrow[r, "r"] \arrow[d, swap, "r"] & PY \arrow[d, "{(t,s)}"] \\
PY \arrow[r, swap, "{(s,t)}"] & Y \times Y
\end{tikzcd}
\end{equation*}
We claim that given $H : X \longrightarrow PY$, (the equivalence class of) $\sigma H$ is the inverse of (the equivalence class of) $H$. It suffices to prove that $\tau (1, \sigma) \simeq_{Y \times Y} rs$ and $\tau (\sigma, 1) \simeq_{Y \times Y} rt$, since composing this with $H$ will give the desired result. We will only show the first equivalence, by verifying that both $\tau (1, \sigma)$ and $rs$ fill the square
\begin{equation*}
\begin{tikzcd}
Y \arrow[r, "r"] \arrow[d, swap, "r"] & PY \arrow[d, "{(s,t)}"] \\
PY \arrow[r, swap, "{(s,s)}"] & Y \times Y
\end{tikzcd}
\end{equation*}
For $rs$, we leave it to the reader to verify that the diagram commutes strictly. We calculate
\begin{align*}
(s,t) \tau (1, \sigma) & = (s \pi_{1}, t \pi_{2}) (1, \sigma) \\
& = (s, t \sigma) \\
& = (s,s)
\end{align*}
and
\begin{align*}
\tau (1, \sigma) r & = \tau (r, \sigma r) \\
& \simeq_{Y \times Y} \tau (r, r) \\
& \simeq_{Y \times Y} r,
\end{align*}
since $\sigma r \simeq_{Y \times Y} r$ while $\tau$ respects $\simeq_{Y \times Y}$.

\subsubsection{Horizontal composition}

\paragraph{Vertical functoriality for whiskering}

To define horizontal composition in the groupoid enriched category $\mathcal{C}$, we firstly define two whiskering operations. Let $f : Y \longrightarrow Z$ be a morphism in $\mathcal{C}$. We show that composing with $f$ extends to a functor $f*-: \mathcal{C}(X,Y) \longrightarrow \mathcal{C}(X,Z)$. Take $Pf$ to be a filler of the square
\begin{equation*}
\begin{tikzcd}
Y \arrow[r, "rf"] \arrow[d, swap, "r"] & PZ \arrow[d, "{(s,t)}"] \\
PY \arrow[r, swap, "{(fs,ft)}"] & Z \times Z
\end{tikzcd}
\end{equation*}
We define the action of $f*-$ on (the equivalence class of) $H : X \longrightarrow PY$ as (the equivalence class of) $PfH$. This is well-defined, by Corollary \ref{cor1} and does not depend on the choice of filler $Pf$.

It is clear that $f*-$ preserves identities, since $(Pf)r \simeq_{Z \times Z} rf$. We must also show that $f*-$ preserves composition. Given $H, K:X \longrightarrow PY$, we need to prove that $(Pf) \tau (H, K) \simeq_{Z \times Z} \tau ( PfH, PfK)$. It suffices to prove that $(Pf) \tau \simeq_{Z \times Z} \tau ( (Pf)\pi_{1}, (Pf) \pi_{2})$ by verifying that both fill the square
\begin{equation*}
\begin{tikzcd}[column sep=large]
Y \arrow[r, "rf"] \arrow[d, swap, "{(r,r)}"] & PZ \arrow[d, "{(s,t)}"] \\
PY \tensor[_t]{\times}{_s} PY \arrow[r, swap, "{(fs \pi_{1},ft \pi_{2})}"] & Z \times Z
\end{tikzcd}
\end{equation*}
Indeed,
\begin{align*}
(s,t)(Pf) \tau & = (fs, ft) \tau \\
& = (f \pi_{1}, f \pi_{2}) (s,t) \tau \\
& = (f \pi_{1}, f \pi_{2}) (s \pi_{1}, t \pi_{2}) \\
& = (fs \pi_{1}, ft \pi_{2})
\end{align*}
and
\begin{align*}
(Pf) \tau (r,r) & \simeq_{Z \times Z} (Pf)r \\
& \simeq_{Z \times Z} r,
\end{align*}
using that $\tau (r,r) \simeq_{Y \times Y} r$, while $Pf$ converts $\simeq_{Y \times Y}$ into $\simeq_{Z \times Z}$. We also see
\begin{align*}
(s,t) \tau ( (Pf)\pi_{1}, (Pf) \pi_{2}) & = (s \pi_{1}, t \pi_{2})( (Pf)\pi_{1}, (Pf) \pi_{2}) \\
& = ( s (Pf)\pi_{1}, t (Pf) \pi_{2}) \\
& = (fs \pi_{1}, ft \pi_{2})
\end{align*}
and
\begin{align*}
\tau ( (Pf)\pi_{1}, (Pf) \pi_{2}) (r,r) & = \tau ( (Pf)r, (Pf)r) \\
& \simeq_{Z \times Z} \tau ( rf, rf) \\
& = \tau (r,r) f \\
& \simeq_{Z \times Z} rf,
\end{align*}
using that $(Pf)r \simeq_{Z \times Z} rf$, while $\tau$ respects $\simeq_{Z \times Z}$.

Now let $g : X \longrightarrow Y$. Composing with $g$ can be extended to a functor $-*g: \mathcal{C}(Y,Z) \longrightarrow \mathcal{C}(X,Z)$. We define the action of $-*g$ on (the equivalence class of) $H : Y \longrightarrow PZ$ as (the equivalence class of) $Hg$. This is clearly well-defined and the fact that this is a functor follows directly from the definitions.

\paragraph{Horizontal functoriality for whiskering}

We claim that for a fixed object $W$ of $\mathcal{C}$, we have a functor $\mathcal{C} \longrightarrow \mathsf{Gpd}$ which sends an object $X$ to $\mathcal{C}(W,X)$ and an arrow $f : X \longrightarrow Y$ to $f*- : \mathcal{C}(W,X) \longrightarrow \mathcal{C}(W,Y)$. Clearly $1*- = 1$, since we may choose $P1  = 1$ as a filler. The second thing we need to check is that for $f : X \longrightarrow Y$ and $g  : Y \longrightarrow Z$ it holds that $(gf)*- = (g*-)(f*-)$. It suffices to prove that $(Pgf) \simeq_{Z \times Z} PgPf$ by verifying that $PgPf$ fills the square
\begin{equation*}
\begin{tikzcd}
X \arrow[r, "rgf"] \arrow[d, swap, "r"] & PZ \arrow[d, "{(s,t)}"] \\
PX \arrow[r, swap, "{(gfs,gft)}"] & Z \times Z
\end{tikzcd}
\end{equation*}
Indeed,
\begin{align*}
(s,t)PgPf & = (gs,gt)Pf \\
& = (g \pi_{1}, g \pi_{2}) (s, t) Pf \\
& = (g \pi_{1}, g \pi_{2}) (fs,ft) \\
& = (gfs,gft)
\end{align*}
and
\begin{align*}
Pg (Pf)r & \simeq_{Z \times Z} (Pg)r \\
& \simeq_{Z \times Z} r,
\end{align*}
using that $(Pf)r \simeq_{Y \times Y} r$, while $Pg$ converts $\simeq_{Y \times Y}$ into $\simeq_{Z \times Z}$.

We also claim that for a fixed object $Z$ of $\mathcal{C}$, we also have a functor $\mathcal{C}^{\mathsf{op}} \longrightarrow \mathsf{Gpd}$ which sends an object $X$ to $\mathcal{C}(X,Z)$ and an arrow $f : X \longrightarrow Y$ to $-*f : \mathcal{C}(Y,Z) \longrightarrow \mathcal{C}(X,Z)$. In this case it follows directly from the definitions that this defines a functor.

Lastly, let $f : W \longrightarrow X$ and $g : Y \longrightarrow Z$. Then $(g*-)(-*f) = (-*f)(g*-)$. This follows directly from the definitions as well.

\paragraph{Two equivalent definitions of horizontal composition}

Let $\alpha : f \longrightarrow f'$ be an arrow in the groupoid $\mathcal{C}(X,Y)$ and let $\beta : g \longrightarrow g'$ be an arrow in the groupoid $\mathcal{C}(Y,Z)$. We define the horizontal composition of $\alpha$ and $\beta$ by $\beta *\alpha = (\beta * f') \circ (g *\alpha)$, where $\circ$ is composition in the groupoid $\mathcal{C}(X,Z)$. We could have chosen instead to define $\beta * \alpha = (g' * \alpha) \circ (\beta * f)$. In fact, these two definitions are equivalent, as we shall demonstrate now. Let $H : f \simeq f'$ and $K : g \simeq g'$ be representatives of $\alpha$ and $\beta$ respectively. Then we need to show that $\tau  (PgH, Kf') \simeq_{Z \times Z} \tau( Kf, Pg'H)$. We claim that $\tau(Ks,Pg')$ fills the square
\begin{equation*}
\begin{tikzcd}
Y \arrow[r, "'K"] \arrow[d, swap, "r"] & PZ \arrow[d, "{(s,t)}"] \\
PY \arrow[r, swap, "{(gs,g't)}"] & Z \times Z
\end{tikzcd}
\end{equation*}
We calculate
\begin{align*}
(s,t)\tau(Ks,Pg') & = (s \pi_{1}, t \pi_{2})(Ks,Pg') \\
& = (sKs, tPg') \\
& = (gs, g't)
\end{align*}
and
\begin{align*}
\tau(Ks,Pg')r & = \tau(Ksr,(Pg')r) \\
& = \tau(K,(Pg')r) \\
& \simeq_{Z \times Z} \tau (K, rg') \\
& \simeq_{Z \times Z} K,
\end{align*}
using that $(Pg')r \simeq_{Z \times Z} rg'$, while $\tau$ respects $\simeq_{Z \times Z}$, as well as the fact that $rg'$ acts as $1_{g'}$. A similar calculation shows that $\tau(Pg,Kt)$ fills the square as well, from which we conclude that $\tau(Ks,Pg') \simeq_{Z \times Z} \tau(Pg,Kt)$. Composing with $H$ gives us the desired equivalence.

\paragraph{Interchange law}

The interchange law is a formal consequence of the laws we have already established. Let $\alpha : f \longrightarrow f'$ and $\alpha' : f' \longrightarrow f''$ be arrows in the groupoid $\mathcal{C}(X,Y)$ and let $K : g \longrightarrow g'$ and $K' : g' \longrightarrow g''$ be arrows in the groupoid $\mathcal{C}(Y,Z)$. We can simply calculate
\begin{align*}
& (\beta' * \alpha') \circ (\beta * \alpha) & \\
= & ((\beta' * f'') \circ (g' * \alpha')) \circ ((\beta * f') \circ (g * \alpha)) & \text{by (the first) definition of horizontal composition} \\
= & ((\beta' * f'') \circ ((g' * \alpha') \circ (\beta * f'))) \circ (g * \alpha) & \text{by associativity of $\circ$} \\ 
= & ((\beta' * f'') \circ (\beta * \alpha')) \circ (g * \alpha) & \text{by (the second) definition of horizontal composition} \\ 
= & ((\beta' * f'') \circ ((\beta * f'') \circ (g * \alpha')) \circ (g * \alpha) & \text{by (the first) definition of horizontal composition} \\ 
= & ((\beta' * f'') \circ (\beta * f'')) \circ ((g * \alpha') \circ (g * \alpha)) & \text{by associativity of $\circ$} \\
= & ((\beta' \circ \beta) * f'') \circ (g * (\alpha' \circ \alpha)) & \text{by functoriality of $-*f$ and $g*-$} \\
= & (\beta' \circ \beta) * (\alpha' \circ \alpha) & \text{by (the first) definition of horizontal composition.}
\end{align*}

\paragraph{Horizontal associativity}

Lastly, we need to show that horizontal composition is associative. Let $\alpha : f \longrightarrow f'$, $\beta : g \longrightarrow g'$ and $\gamma : h \longrightarrow h'$ be arrows in $\mathcal{C}(Y,Z)$, $\mathcal{C}(X,Y)$ and $\mathcal{C}(W,X)$, respectively. We calculate
\begin{align*}
& (\alpha*\beta)*\gamma & \\
= & (((\alpha*g') \circ (f*\beta))*h') \circ ((fg)*\gamma) & \text{by (the first) definition of horizontal composition} \\
= & (((\alpha*g')*h') \circ ((f*\beta)*h')) \circ ((fg)*\gamma) & \text{by functoriality of $-*h$} \\
= & ((\alpha*(g'h')) \circ (f*(\beta*h'))) \circ (f*(g*\gamma)) & \text{by horizontal functoriality for whiskering} \\
= & (\alpha*(g'h')) \circ ((f*(\beta*h')) \circ (f*(g*\gamma))) & \text{by associativity of $\circ$} \\
= & (\alpha*(g'h')) \circ (f*((\beta*h') \circ (g*\gamma))) & \text{by functoriality of $f*-$} \\
= & \alpha*(\beta*\gamma) & \text{by (the first) definition of horizontal composition}
\end{align*}

This finishes the proof that path categories are enriched over groupoids, so at this point we have a function $\mathfrak{E} : \mathsf{Path} \longrightarrow \mathsf{Gpd-Cat}$.

\subsection{Action on 1-cells}

Next, we define the action of $\mathfrak{E}$ on 1-cells. This amounts to showing that a homotopical functor $F : \mathcal{C} \longrightarrow \mathcal{D}$ between path categories can be extended to a 2-functor between the enrichments of $\mathcal{C}$ and $\mathcal{D}$.

\subsubsection{Vertical functoriality for 2-functors}

We claim that for any two objects $X$, $Y$ of $\mathcal{C}$, we can extend the action of $F$ on maps to a functor $F_{X,Y} : \mathcal{C}(X,Y) \longrightarrow \mathcal{D}(FX,FY)$. Let $\lambda_{F}$ be a filler of the following square.
\begin{equation*}
\begin{tikzcd}
FY \arrow[r, "r"] \arrow[d, swap, "Fr"] & PFY \arrow[d, "{(s,t)}"] \\
FPY \arrow[r, swap, "{(Fs, Ft)}"] & FY \times FY
\end{tikzcd}
\end{equation*}
We define $F$ applied to (the equivalence class of) the homotopy $H : X \longrightarrow PY$ to be (the equivalence class of) the homotopy $\lambda_{F} FH : FX \longrightarrow PFY$. This is well-defined, for if we are given given a homotopy $\mathcal{H} : X \longrightarrow P_{Y \times Y} PY$ between $H, H' : X \longrightarrow PY$, then composing $F\mathcal{H}$ with a filler for the square
\begin{equation*}
\begin{tikzcd}[column sep=huge]
FPY \arrow[r, "r\lambda_{F}"] \arrow[d, swap, "Fr"] & P_{FY \times FY} PFY \arrow[d, "{(s,t)}"] \\
FP_{Y \times Y}PY \arrow[r, swap, "{(\lambda_{F} Fs, \lambda_{F} Ft)}"] & PFY \times_{FY \times FY} PFY
\end{tikzcd}
\end{equation*}
results in a homotopy $\lambda_{F} FH \simeq_{FY \times FY} \lambda_{F} FH'$. Also note that this definition does not depend on the choice of $\lambda_{F}$.

Clearly $F_{X,Y}1 = 1$, since $\lambda_{F} Fr \simeq_{FY \times FY} r$. To show that $F_{X,Y}$ preserves composition, it suffices to prove that $\lambda_{F} F \tau \simeq_{FY \times FY} \tau (\lambda_{F} F \pi_{1}, \lambda_{F} F \pi_{2} )$, by verifying that both fill the square
\begin{equation*}
\begin{tikzcd}[column sep=huge]
FY \arrow[r, "r"] \arrow[d, swap, "{F(r,r)}"] & PFY \arrow[d, "{(s,t)}"] \\
F (PY \tensor[_t]{\times}{_s} PY) \arrow[r, swap, "{( F(s \pi_{1}), F(t \pi_{2}))}"] & FY \times FY
\end{tikzcd}
\end{equation*}
Indeed,
\begin{align*}
(s,t)\lambda_{F} F \tau & = (Fs, Ft) F \tau \\
& = (Fs F \tau, Ft F \tau) \\
& = (F (s \tau), F (t \tau)) \\
& = (F (s \pi_{1}), F (t \pi_{2}))
\end{align*}
and
\begin{align*}
\lambda_{F} F \tau F(r,r) & = \lambda_{F} F \tau (r,r) \\
& \simeq_{FY \times FY} \lambda_{F} F r \\
& \simeq_{FY \times FY} r,
\end{align*}
using that $\tau (r,r) \simeq_{Y \times Y} r$, while $\lambda_{F} F$ converts $\simeq_{Y \times Y}$ into $\simeq_{FY \times FY}$. Furthermore
\begin{align*}
(s,t)\tau (\lambda_{F} F \pi_{1}, \lambda_{F} F \pi_{2}) & = (s \pi_{1}, t \pi_{1}) (\lambda_{F} F \pi_{1}, \lambda_{F} F \pi_{2}) \\
& = (s \lambda_{F} F \pi_{1}, t \lambda_{F} F \pi_{2}) \\
& = (F s F \pi_{1}, F t F \pi_{2}) \\
& = (F (s \pi_{1}), F (t \pi_{2}))
\end{align*}
and
\begin{align*}
\tau (\lambda_{F} F \pi_{1}, \lambda_{F} F \pi_{2}) F(r,r) &= \tau (\lambda_{F} F r, \lambda_{F} F r) \\
& \simeq_{FY \times FY} \tau (r,r) \\
& \simeq_{FY \times FY} r,
\end{align*}
using that $\lambda_{F} F r \simeq_{FY \times FY} r$, while $\tau$ respects $\simeq_{FY \times FY}$.

\subsubsection{Horizontal functoriality for 2-functors}

\paragraph{Preservation of whiskering}

To show that $F : \mathcal{C} \longrightarrow \mathcal{D}$ respects horizontal composition, we firstly show that $F$ respects whiskering.  Let $X$, $Y$ and $Z$ be objects of $\mathcal{C}$ and let $f : Y \longrightarrow Z$ be a morphism. We claim that the following square commutes
\begin{equation*}
\begin{tikzcd}
\mathcal{C}(X,Y) \arrow[r, "F"] \arrow[d, swap, "f*-"] & \mathcal{D}(FX,FY) \arrow[d, "(Ff)*-"] \\
\mathcal{C}(X,Z) \arrow[r, swap, "F"] & \mathcal{D}(FX,FZ)
\end{tikzcd}
\end{equation*}
To show this, it suffices to prove that $\lambda_{F} FPf \simeq_{FZ \times FZ} (PFf) \lambda_{F}$, by verifying that both fill the square
\begin{equation*}
\begin{tikzcd}
FY \arrow[r, "rFf"] \arrow[d, swap, "Fr"] & PFZ \arrow[d, "{(s,t)}"] \\
FPY \arrow[r, swap, "{(Ffs,Fft)}"] & FZ \times FZ
\end{tikzcd}
\end{equation*}
We calculate
\begin{align*}
(s,t) \lambda_{F} FPf & = (Fs, Ft) FPf \\
& = (FsFPf, FtFPf) \\
& = (FsPf, FtPf) \\
& = (Ffs, Fft)
\end{align*}
and
\begin{align*}
\lambda_{F} FPf Fr & = \lambda_{F} FPfr \\
& \simeq_{FZ \times FZ} \lambda_{F} Frf \\
& = \lambda_{F} Fr Ff \\
& \simeq_{FZ \times FZ} rFf,
\end{align*}
using that $Pfr \simeq_{Z \times Z} rf$, while $\lambda_{F} F$ converts $\simeq_{Z \times Z}$ into $\simeq_{FZ \times FZ}$. We also see
\begin{align*}
(s,t) PFf \lambda_{F} & = (Ffs, Fft) \lambda_{F} \\
& = (Ffs \lambda_{F}, Fft \lambda_{F}) \\
& = (FfFs, FfFt) \\
& = (Ffs, Fft)
\end{align*}
and
\begin{align*}
PFf \lambda_{F} Fr & \simeq_{FZ \times FZ} PFfr \\
& \simeq_{FZ \times FZ} rFf,
\end{align*}
using that $\lambda_{F} Fr \simeq_{FY \times FY} r$, while $PFf$ converts $\simeq_{FY \times FY}$ into $\simeq_{FZ \times FZ}$.

Now let $g : X \longrightarrow Y$. We also claim that the following square commutes
\begin{equation*}
\begin{tikzcd}
\mathcal{C}(Y,Z) \arrow[r, "F"] \arrow[d, swap, "-*g"] & \mathcal{D}(FY,FZ) \arrow[d, "-*(Fg)"] \\
\mathcal{C}(X,Z) \arrow[r, swap, "F"] & \mathcal{D}(FX,FZ)
\end{tikzcd}
\end{equation*}
In this case the commutativity of the square follows directly from the definitions.

\paragraph{Preservation of horizontal composition}

Let $X$, $Y$ and $Z$ be objects of $\mathcal{C}$. We claim that the following square commutes
\begin{equation*}
\begin{tikzcd}
\mathcal{C}(Y,Z) \times \mathcal{C}(X,Y) \arrow[r, "-*-"] \arrow[d, swap, "F \times F"] & \mathcal{C}(X,Z) \arrow[d, "F"] \\
\mathcal{D}(FY,FZ) \times \mathcal{D}(FX,FY) \arrow[r, swap, "-*-"] & \mathcal{D}(FX,FZ)
\end{tikzcd}
\end{equation*}

Let $\alpha : f \longrightarrow f'$ and $\beta : g \longrightarrow g'$ be arrows in $\mathcal{C}(Y,Z)$, $\mathcal{C}(X,Y)$ and $\mathcal{C}(W,X)$, respectively. We calculate
\begin{align*}
& F(\alpha*\beta) & \\
= & F((\alpha*g') \circ (f*\beta)) & \text{by definition of horizontal composition} \\
= & F(\alpha*g') \circ F(f*\beta) & \text{by functoriality of $F_{X,Z}$} \\
= & (F\alpha*Fg') \circ (Ff*F\beta) & \text{by preservation of whiskering} \\
= & F\alpha*F\beta & \text{by definition of horizontal composition}
\end{align*}

\subsubsection{Functoriality on 1-cells}

We show that $\mathfrak{E}$ respects horizontal composition and identity. Note that if $F = 1$, then $F_{X,Y} = 1$, since we may choose $\lambda_{F} = 1$. Now suppose that we have a second homotopical functor $G : \mathcal{D} \longrightarrow \mathcal{E}$. We claim that $G_{FX, FY} F_{X,Y}  = (FG)_{X,Y}$. It suffices to show that $\lambda_{GF} \simeq_{GFY \times GFY} \lambda_{G} G \lambda_{F}$ by verifying that $\lambda_{G} G \lambda_{F}$ fills the square
\begin{equation*}
\begin{tikzcd}[column sep=huge]
GFY \arrow[r, "r"] \arrow[d, swap, "GFr"] & PGFY \arrow[d, "{(s,t)}"] \\
GFPY \arrow[r, swap, "{(GFs, GFt)}"] & GFY \times GFY
\end{tikzcd}
\end{equation*}
We calculate
\begin{align*}
(s,t) \lambda_{G} G \lambda_{F} & = (Gs, Gt) G \lambda_{F} \\
& = (GsG\lambda_{F}, GtG\lambda_{F}) \\
& = (Gs\lambda_{F}, Gt\lambda_{F}) \\
& = (GFs, GFt)
\end{align*}
and
\begin{align*}
\lambda_{G} G \lambda_{F} GFr & = \lambda_{G} G \lambda_{F} Fr \\
& \simeq_{GFY \times GFY} \lambda_{G} G r \\
& \simeq_{GFY \times GFY} r,
\end{align*}
using that $\lambda_{F} Fr \simeq_{FY \times FY} r$, while $\lambda_{G} G$ converts $\simeq_{FY \times FY}$ into $\simeq_{GFY \times GFY}$.

This finishes the part concerning the 1-cells, so at this point we have a functor $\mathfrak{E} : \mathsf{Path} \longrightarrow \mathsf{Gpd-Cat}$.

\subsection{Action on 2-cells}

Let $F, G : \mathcal{C} \longrightarrow \mathcal{D}$ be homotopical functors between path categories. Let $X$ and $Y$ be objects of $\mathcal{C}$ and let $\alpha : F \longrightarrow G$ be a natural transformation. We claim that $\alpha$ induces a strict transformation between $F$ and $G$, as 2-functors. It suffices to show that the following square commutes
\begin{equation*}
\begin{tikzcd}
\mathcal{C}(X,Y) \arrow[r, "F"] \arrow[d, swap, "G"] & \mathcal{D}(FX,FY) \arrow[d, "\alpha_{Y}*-"] \\
\mathcal{C}(GX,GY) \arrow[r, swap, "-*\alpha_{X}"] & \mathcal{D}(FX,GY)
\end{tikzcd}
\end{equation*}
To show this, it suffices to prove that $\lambda_{G} \alpha_{PY} \simeq_{GY \times GY} (P\alpha_{Y}) \lambda_{F}$, by verifying that both fill the square
\begin{equation*}
\begin{tikzcd}[column sep = huge]
FY \arrow[r, "r\alpha_{Y}"] \arrow[d, swap, "Fr"] & PGY \arrow[d, "{(s,t)}"] \\
FPY \arrow[r, swap, "{(\alpha_{Y} Fs, \alpha_{Y} Fft)}"] & GY \times GY
\end{tikzcd}
\end{equation*}
We calculate
\begin{align*}
(s,t) \lambda_{G} \alpha_{PY} & = (Gs, Gt) \alpha_{PY} \\
& = (Gs \alpha_{PY}, Gt \alpha_{PY}) \\
& = (\alpha_{Y} Fs, \alpha_{Y} Ft)
\end{align*}
and
\begin{align*}
\lambda_{G} \alpha_{PY} Fr & = \lambda_{G} Gr \alpha_{Y} \\
& \simeq_{GY \times GY} r \alpha_{Y}
\end{align*}
Furthermore,
\begin{align*}
(s,t) (P \alpha_{Y}) \lambda_{F} & = (\alpha_{Y} s, \alpha_{Y} t) \lambda_{F} \\
& = (\alpha_{Y} s \lambda_{F}, \alpha_{Y} t \lambda_{F}) \\
& = (\alpha_{Y} Fs, \alpha_{Y} Ft)
\end{align*}
and
\begin{align*}
(P \alpha_{Y}) \lambda_{F} Fr & \simeq_{GY \times GY} (P \alpha_{Y}) r \\
& \simeq_{GY \times GY} r \alpha_{Y}
\end{align*}
using that $\lambda_{F} Fr \simeq_{FY \times FY} r$, while $P \alpha_{Y}$ converts $\simeq_{FY \times FY}$ into $\simeq_{GY \times GY}$. It is clear that $\mathfrak{E}$ respects vertical composition and identity.

This concludes the construction of the 2-functor $\mathfrak{E} : \mathsf{Path} \longrightarrow \mathsf{Gpd-Cat}$.

\subsection{Partial enrichment for weak path categories}

In studying some aspects of a path category $\mathcal{C}$, we cannot avoid to work in the slices $\mathcal{C}/X$ of $\mathcal{C}$ as well. Such a slice $\mathcal{C}/X$ is almost, but not exactly, a path category itself, with the fibrations and weak equivalences inherited from $\mathcal{C}$. The problem with $\mathcal{C}/X$ is that not all of its objects are fibrant. We will call such a structure a \textit{weak path category}.
\begin{dfn}
A category $\mathcal{C}$ is a \textit{weak path category} if the following axioms are satisfied.
\begin{enumerate}
\item{Fibrations are closed under composition.}
\item{The pullback of a fibration along any other map exists and is again a fibration.}
\item{The pullback of an acyclic fibration along any other map is again an acyclic fibration}
\item{Weak equivalences satisfy 2-out-of-6.}
\item{Isomorphisms are acyclic fibrations and every acyclic fibration has a section.}
\item{$\mathcal{C}$ has a terminal object.}
\item{For every fibrant object $X$, there is a path object $PX$.}
\end{enumerate}
\end{dfn}
Write $\mathsf{WPath}$ for the (2-)category of weak path categories and fibrant object preserving homotopical functors. We would like to construct a 2-functor $\mathfrak{E} : \mathsf{Wpath} \longrightarrow \mathsf{Gpd-Cat}$ as before. Unfortunately, this is a bit too much to ask. Given two objects $X$ and $Y$ of a weak path category $\mathcal{C}$, where $Y$ fails to be fibrant, we can no longer turn $\mathcal{C}(X,Y)$ into a groupoid as we did before, since $Y$ may not have a path object. However, the full subcategory $\mathcal{C}_{\mathsf{f}}$, consisting of the fibrant objects of $\mathcal{C}$, \textit{is} a path category. This fact allows us to salvage the results of the previous section we can reasonably expect to hold for weak path categories.

The key observation is that the constructions we have made are all \textit{internal}. For example, in the process of proving that for objects $X$ and $Y$ of a path category $\mathcal{C}$ we can give $\mathcal{C}(X,Y)$ the structure of a groupoid, we have actually constructed something we could call an \textit{internal E-groupoid}. By this we mean that we have shown that we have an object of objects $Y$, an object of arrows $PY$ and an object of equivalences $P_{Y \times Y}PY$, along with various morphisms expressing that $P_{Y \times Y}PY$ is an equivalence relation on parallel arrows and that $Y$ and $PY$ satisfy the groupoid axioms, up to this equivalence. The existence of an internal E-groupoid, with $Y$ as object of objects, is sufficient to give $\mathcal{C}(X,Y)$ the structure of a groupoid for any $X$. The significance of this observation is that if $F : \mathcal{C} \longrightarrow \mathcal{D}$ is a functor into an arbitrary category $\mathcal{D}$ which preserves the terminal object and pullbacks along fibrations, then the structure of the internal E-groupoid is preserved by $F$. It follows that given an object $X$ of $\mathcal{D}$ and an object $Y$ of $\mathcal{C}$, we can give the hom-set $\mathcal{D}(X, FY)$ the structure of a groupoid, using the internal E-groupoid consisting of $FY$, $FPY$ and $FP_{Y \times Y}PY$. The same reasoning applies to the other constructions we have made. Note that given a weak path category $\mathcal{C}$, the inclusion $\mathcal{C}_{\mathsf{f}} \longrightarrow \mathcal{C}$ has the aforementioned properties.

Write $\mathsf{Gpd}$ for the large category of small groupoids and write $\mathsf{CAT}$ for the very large 2-category of large categories. Recall that an object of the colax slice category $\mathsf{CAT} \sslash \mathsf{Gpd}$ is a functor $X : \mathcal{C} \longrightarrow \mathsf{Gpd}$ and a morphism from $X : \mathcal{C} \longrightarrow \mathsf{Gpd}$ to $Y : \mathcal{D} \longrightarrow \mathsf{Gpd}$ consists of a functor $F : \mathcal{C} \longrightarrow \mathcal{D}$ together with a natural transformation $\alpha : X \longrightarrow YF$. Then our observation gives us the following Lemma.
\begin{lem} \label{mainlem}
The constructions used for the 2-functor $\mathfrak{E} : \mathsf{Path} \longrightarrow \mathsf{Gpd-Cat}$ also yield a functor $\mathsf{WPath} \longrightarrow \mathsf{CAT} \sslash \mathsf{Gpd}$ sending a weak path category $\mathcal{C}$ to the functor $\mathcal{C}(-,-) : \mathcal{C}^{\mathsf{op}} \times \mathcal{C}_{\mathsf{f}} \longrightarrow \mathsf{Gpd}$.
\end{lem}
Lemma \ref{mainlem} does not fully capture all of the results that carry over to the weak setting, but is merely a compact formulation of those results that we need in the parts that follow.

\subsection{Preservation of fibrations and weak equivalences}

We show that, in some sense, the previous constructions turn fibrations and weak equivalences in a path category $\mathcal{C}$ into fibrations and weak equivalences in $\mathsf{Gpd}$.

\begin{lem} \label{fiblem}
Let $f : Y \longrightarrow Z$ be a fibration in a path category $\mathcal{C}$. Then the functor $f*- : \mathcal{C}(X, Y) \longrightarrow \mathcal{C}(X, Z)$ is an isofibration, for every $X$.
\end{lem}

\begin{proof}
If we spell out what it means for $f*-$ to be an isofibration, then we see that given $g : X \longrightarrow Y$ and $H : fg \simeq h$, we need to find some $K : f \simeq f'$ such that $PfK \simeq_{Z \times Z} H$. By Theorem 2.28 of \cite{Path}, we can choose $PY$ and $Pf$ in such a way that there exists a map $\nabla : Y \tensor[_f]{\times}{_s} PZ \longrightarrow PY$ with the property that $(Pf)\nabla = \pi_{2}$. We take $K = \nabla (g, H)$ to get
\begin{align*}
PfK & = (Pf) \nabla (g, H) \\
& =\pi_{2} (f, H) \\
& = H,
\end{align*}
as required.
\end{proof}

\begin{lem} \label{homlem}
Let $f, g : X \longrightarrow Y$ be homotopic maps in a path category $\mathcal{C}$. Then the functors $f*-, g*-: \mathcal{C}(W, X) \longrightarrow \mathcal{C}(W, Y)$ are naturally isomorphic, for every $W$. Similarly, the functors $-*f, -*g: \mathcal{C}(Y, Z) \longrightarrow \mathcal{C}(X, Z)$ are naturally isomorphic, for every $Z$.
\end{lem}

\begin{proof}
The fact that $f$ and $g$ are homotopic simply means that there exists an arrow $\alpha : f \longrightarrow g$ in $\mathcal{C}(X, Y)$.  Such an arrow induces the desired natural isomorphisms in any groupoid enriched category.
\end{proof}

\begin{cor} \label{coreq}
Let $f : X \longrightarrow Y$ be be a weak equivalence in a path category $\mathcal{C}$. Then the functors $f*-: \mathcal{C}(W, X) \longrightarrow \mathcal{C}(W, Y)$ and $-*f : \mathcal{C}(Y, Z) \longrightarrow \mathcal{C}(X, Z)$ are equivalences of categories, for every $W$ and $Z$.
\end{cor}

\section{Homotopy function spaces} \label{sec4}

We formulate the homotopy universal properties that our function spaces should satisfy in terms of functor properties. We say that a functor $F : \mathcal{C} \longrightarrow \mathcal{D}$ is \textit{essentially injective} if two objects $X$ and $Y$ of $\mathcal{C}$ are isomorphic whenever $FX$ and $FY$ are. The relevant properties are \textit{essentially surjective, (e.s.)}, \textit{essentially surjective and essentially injective, (e.s.e.i.)} and \textit{essentially surjective and full, (e.s.f.)}. If we limit ourselves to functors between groupoids, then the properties are listed in order of increasing strength. Let us first collect some elementary facts about e.s.(f.) functors.

\begin{lem} \label{esflem}
Let $A$, $B$ and $C$ be groupoids and let $F : A \longrightarrow B$ and $G : B \longrightarrow C$ be functors between them. Put $H = GF$. Then the following implications hold
\begin{itemize}
\item $F$ and $G$ are e.s.(f.) $\Longrightarrow$ $H$ is e.s.(f.)
\item $F$ and $H$ are e.s.(f.) $\Longrightarrow$ $G$ is e.s.(f.)
\end{itemize}
\end{lem}

\begin{lem} \label{1lem}
Let $A$, $B$ and $C$ be groupoids and let $F : A \longrightarrow B$ and $G : B \longrightarrow C$ be functors between them. Suppose that $H = GF$ is an isofibration. Then $F$ is e.s.(f.) if and only if the map $\alpha^{*}A \longrightarrow \alpha^{*}B$ induced by $F$ is e.s.(f.) for every map $\alpha : 1 \longrightarrow C$.
\end{lem}

\begin{lem} \label{esfpblem}
Let $A$, $B$ and $C$ be groupoids and let $F : A \longrightarrow C$ and $G : B \longrightarrow C$ be functors between them. Suppose that $F$ is e.s.(f.) and $G$ is an isofibration. Then the pullback of $F$ along $G$ is e.s.(f.) as well.
\end{lem}

\begin{dfn} \label{expdef}
Let $X$ and $Y$ be objects of a path category $\mathcal{C}$. A \textit{weak homotopy exponential} (resp. \textit{homotopy exponential}, resp. \textit{strong homotopy exponential}) for $X$ and $Y$ consists of an object $Y^{X}$ and a map $\varepsilon_{Y} : Y^{X} \times X \longrightarrow Y$, such that for every object $T$, the functor
\begin{equation*}
\begin{tikzcd}
\mathcal{C}(T,Y^{X}) \arrow[r, "- \times X"] & \mathcal{C}(T \times X, Y^{X} \times X) \arrow[r, "\varepsilon_{Y}*-"] & \mathcal{C}(T \times X, Y)
\end{tikzcd}
\end{equation*}
is e.s. (resp. e.s.e.i., resp. e.s.f.).
\end{dfn}

\begin{dfn} \label{pidef}
Let $f : X \longrightarrow I$ and $g : I \longrightarrow J$ be fibrations in a path category $\mathcal{C}$. A \textit{weak homotopy $\Pi$-type} (resp. \textit{homotopy $\Pi$-type}, resp. \textit{strong homotopy $\Pi$-type}) for $f$ and $g$ consists of a fibration $\Pi_{g} f : \Pi_{g} X \longrightarrow J$ and a map $\varepsilon_{X} : (\Pi_{g} X) \times_{J} I  \longrightarrow X$, over $I$, such that for every map $T \longrightarrow J$, the functor
\begin{equation*}
\begin{tikzcd}
(\mathcal{C}/J)(T,\Pi_{g} X) \arrow[r, "- \times_{J} I"] & (\mathcal{C}/I)(T \times_{J} I, (\Pi_{g} X) \times_{J} I ) \arrow[r, "\varepsilon_{X}*-"] & (\mathcal{C}/I)(T \times_{J} I, X)
\end{tikzcd}
\end{equation*}
is e.s. (resp. e.s.e.i., resp. e.s.f.).
\end{dfn}

We could introduce an even stronger notion of homotopy exponentials, corresponding to essentially surjective and fully faithful functors, but we will show in Corollary \ref{verystr} that this stronger property is automatically satisfied by strong homotopy exponentials in the presence of strong homotopy $\Pi$-types.

We will also refer to homotopy exponentials as \textit{ordinary} homotopy exponentials.  The following three Lemmas show that the notions of weak, ordinary and strong homotopy exponential are stable under weak equivalences.

\begin{lem} \label{stable1}
Let $(Y^{X}, \varepsilon_{Y})$ be a weak (resp. ordinary, resp. strong) homotopy exponential for $X$ and $Y$ in a path category $\mathcal{C}$ and suppose that $h : Z \longrightarrow Y^{X}$ is a weak equivalence. Then $(Z, \varepsilon_{Y}(h \times 1))$ is a weak (resp. ordinary, resp. strong) homotopy exponential for $X$ and $Y$.
\end{lem}

\begin{proof}
Let $T$ be an arbitrary object of $\mathcal{C}$. The diagram
\begin{equation*}
\begin{tikzcd}[column sep = large]
\mathcal{C}(T, Z) \arrow[r, "- \times X"] \arrow[d, swap, "h*-"] & \mathcal{C}(T \times X, Z \times X) \arrow[r, "(\varepsilon_{Y}(h \times 1))*-"] \arrow[d, "(h \times 1)*-"] & \mathcal{C}(T \times X, Y) \\
\mathcal{C}(T, Y^{X}) \arrow[r, swap, "- \times X"] & \mathcal{C}(T \times X, Y^{X} \times X) \arrow[ru, swap, "\varepsilon_{Y} *-"]
\end{tikzcd}
\end{equation*}
commutes by Lemma \ref{mainlem}. The functor $h*-$ is e.s.f., since it is an equivalence of categories by Corollary \ref{coreq}. Since $(\varepsilon_{Y}*-)(- \times X)$ is e.s. (resp. e.s.e.i., resp. e.s.f.) by assumption, this implies that $(\varepsilon_{Y}*-)(- \times X)(h*-)$ is e.s. (resp. e.s.e.i., resp. e.s.f.) by Lemma \ref{esflem}. Hence $((\varepsilon_{Y}(h \times 1))*-)(- \times X)$ is e.s. (resp. e.s.e.i., resp e.s.f.), as needed.
\end{proof}

\begin{lem} \label{stable2}
Let $(Y^{X}, \varepsilon_{Y})$ be a weak (resp. ordinary, resp. strong) homotopy exponential for $X$ and $Y$ in a path category $\mathcal{C}$ and suppose that $h : Y \longrightarrow Z$ is weak equivalence. Then $(Y^{X}, h \varepsilon_{Y})$ is a weak (resp. ordinary, resp. strong) homotopy exponential for $X$ and $Z$.
\end{lem}

\begin{proof}
Let $T$ be an arbitrary object of $\mathcal{C}$. The functor $h*- : \mathcal{C}(T \times X, Y) \longrightarrow \mathcal{C}(T \times X, Z)$ is e.s.f., since it is an equivalence of categories by Corollary \ref{coreq}. Since $(\varepsilon_{Y}*-)(- \times X) : \mathcal{C}(T, Y^{X}) \longrightarrow \mathcal{C}(T \times X, Y)$ is e.s. (resp. e.s.e.i., resp. e.s.f.) by assumption, this implies that $(h*-)(\varepsilon_{Y}*-)(- \times X)$ is e.s. (resp. e.s.e.i., resp. e.s.f.) by Lemma \ref{esflem}, as needed.
\end{proof}

\begin{lem}
Let $(Y^{X}, \varepsilon_{Y})$ be a weak (resp. ordinary, resp. strong) homotopy exponential for $X$ and $Y$ in a path category $\mathcal{C}$ and suppose that $h : W \longrightarrow X$ is weak equivalence. Then $(Y^{X}, \varepsilon_{Y}(h \times 1))$ is a weak (resp. ordinary, resp. strong) homotopy exponential for $W$ and $Y$.
\end{lem}

\begin{proof}
Let $T$ be an arbitrary object of $\mathcal{C}$. The arrow $h : W \longrightarrow X$ gives rise to a natural transformation between the two homotopical functors $- \times W, - \times X : \mathcal{C} \longrightarrow \mathcal{C}$. This, in turn, induces a strict transformation, showing that the left-hand square of the following diagram commutes.
\begin{equation*}
\begin{tikzcd}
\mathcal{C}(T, Y^{X}) \arrow[r, "- \times X"] \arrow[d, swap, "- \times W"] & \mathcal{C}(T \times X, Y^{X} \times X) \arrow[r, "\varepsilon_{Y} * -"] \arrow[d, "- * (h \times 1)"] & \mathcal{C}(T \times X, Y) \arrow[d, "- * (h \times 1)"] \\
\mathcal{C}(T \times W, Y^{X} \times W) \arrow[r, swap, "(h \times 1) * -"] & \mathcal{C}(T \times W, Y^{X} \times X) \arrow[r, swap, "\varepsilon_{Y} * -"] & \mathcal{C}(T \times W, Y)
\end{tikzcd}
\end{equation*}
The right-hand square commutes by the associative law for 1-cells in the groupoid enriched category $\mathcal{C}$. The functor $- * (h \times 1)$ is e.s.f., since it is an equivalence of categories by Corollary \ref{coreq}. Since $(\varepsilon_{Y}*-)(- \times X)$ is e.s. (resp. e.s.e.i., resp. e.s.f.) by assumption, this implies that $(- * (h \times 1))(\varepsilon_{Y}*-)(- \times X)$ is e.s. (resp. e.s.e.i., resp. e.s.f.) by Lemma \ref{esflem}. Hence $(\varepsilon_{Y} * -)((h \times 1) * -)(- \times W)$ is e.s. (resp. e.s.e.i., resp. e.s.f.), as needed.
\end{proof}

\begin{ex}
We can use Lemma \ref{stable2} to give an example of a proper strong homotopy exponential, where the (strict) exponential does not even exist. The category of topological spaces has the structure of a path category, taking the homotopy equivalences as weak equivalences and the Hurewicz fibrations as the fibrations. It is known that the (strict) exponential of the rationals, $\mathbb{Q}$, and the Sierpi\'nski space, $\mathbb{S}$, does not exist (see \cite{Top}). However, $\mathbb{S}$ is contractible and $1$ is clearly a strong homotopy exponential of $\mathbb{Q}$ and $1$. Hence, the strong homotopy exponential of $\mathbb{Q}$ and $\mathbb{S}$ does exist, by Lemma \ref{stable2}.
\end{ex}

We are unable to provide an example of an ordinary homotopy exponential which is not strong, although we do suspect that the two notions are distinct. The examples of ordinary homotopy exponentials that we know of satisfy the properties of strong homotopy exponentials as well. Furthermore, the following Lemma explains why the strategy of kneading an existing strong homotopy exponential into an ordinary homotopy exponential cannot succeed.

\begin{lem}
Let $(Y^{X}, \varepsilon_{Y})$ be a strong homotopy exponential and let $(\widetilde{Y^{X}}, \widetilde{\varepsilon_{Y}})$ be an ordinary homotopy exponential for $X$ and $Y$ in a path category $\mathcal{C}$. Then $(\widetilde{Y^{X}}, \widetilde{\varepsilon_{Y}})$ is a strong homotopy exponential for $X$ and $Y$.
\end{lem}

\begin{proof}
By the homotopy universal property of $Y^{X}$ and $\widetilde{Y^{X}}$, there exists a weak equivalence $h: \widetilde{Y^{X}} \longrightarrow Y^{X}$, such that the diagram
\begin{equation*}
\begin{tikzcd}
\widetilde{Y^{X}} \times X \arrow[r, "\widetilde{\varepsilon_{Y}}"] \arrow[d, swap, "h \times 1"] & Y \\
Y^{X} \times X \arrow[ru, swap, "\varepsilon_{Y}"]
\end{tikzcd}
\end{equation*}
commutes up to homotopy. Let $T$ be an arbitrary object of $\mathcal{C}$. Then by Lemma \ref{homlem} we know that 
\begin{equation*}
\begin{tikzcd}
\mathcal{C}(T \times X, \widetilde{Y^{X}} \times X) \arrow[r, "\widetilde{\varepsilon_{Y}}*-"] \arrow[d, swap, "(h \times 1)*-"] & \mathcal{C}(T \times X, Y) \\
\mathcal{C}(T \times X, Y^{X} \times X) \arrow[ru, swap, "\varepsilon_{Y}*-"]
\end{tikzcd}
\end{equation*}
commutes up to natural isomorphism. Moreover, the diagram
\begin{equation*}
\begin{tikzcd}[column sep = large]
\mathcal{C}(T, \widetilde{Y^{X}}) \arrow[r, "- \times X"] \arrow[d, swap, "h*-"] & \mathcal{C}(T \times X, \widetilde{Y^{X}} \times X) \arrow[d, "(h \times 1)*-"] \\
\mathcal{C}(T, Y^{X}) \arrow[r, swap, "- \times X"] & \mathcal{C}(T \times X, Y^{X} \times X)
\end{tikzcd}
\end{equation*}
commutes by Lemma \ref{mainlem}. The functor $h*-$ is e.s.f., since it is an equivalence of categories by Corollary \ref{coreq}. Since $(\varepsilon_{Y}*-)(- \times X)$ is e.s.f. by assumption, this implies that $(\varepsilon_{Y}*-)(- \times X)(h*-)$ is e.s.f. by Lemma \ref{esflem}. Hence $(\widetilde{\varepsilon_{Y}}*-)(- \times X)$ is e.s.f., since e.s.f. functors are closed under natural isomorphism.
\end{proof}

\section{Results on homotopy function spaces} \label{sec5}

\begin{prop} \label{mainprop1}
Let $\mathcal{C}$ be a path category with weak (resp. strong) homotopy $\Pi$-types. Given a fibration $p : Y \longrightarrow Z$ and a weak (strong) homotopy exponential $(Z^{X}, \varepsilon_{Z})$, there is a weak (strong) homotopy exponential $(Y^{X}, \varepsilon_{Y})$ and a fibration $p^{X} : Y^{X} \longrightarrow Z^{X}$ such that the square
\begin{equation*}
\begin{tikzcd}
Y^{X} \times X \arrow[r, "\varepsilon_{Y}"] \arrow[d, swap, "p^{X} \times 1"] & Y \arrow[d, "p"] \\
Z^{X} \times X \arrow[r, swap, "\varepsilon_{Z}"] & Z
\end{tikzcd}
\end{equation*}
commutes.
\end{prop}

\begin{proof}
Take the pullback
\begin{equation*}
\begin{tikzcd}
Q \arrow[r, "e"] \arrow[d, swap, "q"] & Y \arrow[d, "p"] \\
Z^{X} \times X \arrow[r, swap, "\varepsilon_{Z}"] & Z
\end{tikzcd}
\end{equation*}
and define $p^{X} : Y^{X} \longrightarrow Z^{X}$ as $\Pi_{\pi_{1}}q : \Pi_{\pi_{1}} Q \longrightarrow Z^{X}$, where $\pi_{1} : Z^{X} \times X \longrightarrow Z^{X}$. Put $\varepsilon_{Y} = e \varepsilon_{Q}$. Clearly $p^{X}$ meets our requirements. We must show that $(Y^{X}, \varepsilon_{Y})$ is a weak (strong) homotopy exponential. Let $g : T \longrightarrow Z^{X}$ be arbitrary. Since the square
\begin{equation*}
\begin{tikzcd}
\mathcal{C}/Z^{X} \arrow[r, "- \times X"] \arrow[d] & \mathcal{C}/Z^{X} \times X \arrow[d] \\
\mathcal{C} \arrow[r, swap, "- \times X"] & \mathcal{C}
\end{tikzcd}
\end{equation*}
of homotopical functors commutes, Lemma \ref{mainlem} shows that that the cube
\begin{equation} \label{diag1}
\begin{tikzcd}[row sep = large, column sep = tiny]
& \mathcal{C}(T,Y^{X}) \arrow[rr, "- \times X"] \arrow[dd, near end, swap, "p^{X}*-"] & & \mathcal{C}(T \times X, Y^{X} \times X) \arrow[dd, "(p^{X} \times 1)*-"] \\
(\mathcal{C}/Z^{X})(T,Y^{X}) \arrow[rr, crossing over, near end, "- \times X"] \arrow[dd, swap, "p^{X}*-"] \arrow[ur] & & (\mathcal{C}/Z^{X} \times X)(T \times X, Y^{X} \times X) \arrow[ur] & \\
& \mathcal{C}(T,Z^{X}) \arrow[rr, near start, "- \times X"] & & \mathcal{C}(T \times X, Z^{X} \times X) \\
(\mathcal{C}/Z^{X})(T,Z^{X}) \arrow[rr, "- \times X"] \arrow[ur] & & (\mathcal{C}/Z^{X} \times X)(T \times X, Z^{X} \times X) \arrow[ur] \arrow[from=uu, crossing over, swap, near end, "(p^{X} \times 1)*-"] &
\end{tikzcd}
\end{equation}
commutes. The prism
\begin{equation} \label{diag2}
\begin{tikzcd}[row sep = large, column sep = tiny]
& \mathcal{C}(T \times X, Y^{X} \times X) \arrow[rr, "\varepsilon_{Q}*-"] \arrow[ddr, near start, "( p^{X} \times 1)*-"] & & \mathcal{C}(T \times X, Q) \arrow[ddl, dash pattern=on 45pt off 15pt, near start, "q*-"] \\
(\mathcal{C}/Z^{X} \times X)(T \times X, Y^{X} \times X) \arrow[ur] \arrow[rr, crossing over, "\varepsilon_{Q}*-"] \arrow[ddr, swap, "(p^{X} \times 1)*-"] & & (\mathcal{C}/ Z^{X} \times X) (T \times X, Q) \arrow[ur] \arrow[ddl, crossing over, swap, "q*-"] & \\
& & \mathcal{C}(T \times X, Z^{X} \times X)  & \\
& (\mathcal{C}/Z^{X} \times X)(T \times X, Z^{X} \times X) \arrow[ur] & &
\end{tikzcd}
\end{equation}
also commutes by Lemma \ref{mainlem}. Lastly, consider the cube
\begin{equation} \label{diag3}
\begin{tikzcd}[row sep = large, column sep = tiny]
& \mathcal{C}(T \times X,Q) \arrow[rr, "e*-"] \arrow[dd, near end, swap, "q*-"] & & \mathcal{C}(T \times X, Y) \arrow[dd, "p*-"] \\
(\mathcal{C}/Z^{X} \times X)(T \times X,Q) \arrow[rr, crossing over, near end, "i"] \arrow[dd, swap, "q*-"] \arrow[ur] & & (\mathcal{C}/Z)(T \times X, Y) \arrow[ur] & \\
& \mathcal{C}(T \times X,Z^{X} \times X) \arrow[rr, near start, "\varepsilon_{Z}*-"] & & \mathcal{C}(T \times X, Z) \\
(\mathcal{C}/Z^{X} \times X)(T \times X,Z^{X} \times X) \arrow[rr, "\varepsilon_{Z}*-"] \arrow[ur] & & (\mathcal{C}/Z)(T \times X, Z) \arrow[ur] \arrow[from=uu, crossing over, swap, near end, "p*-"] &
\end{tikzcd}
\end{equation}
The top of the cube (\ref{diag3}) commutes by Lemma \ref{pblem} (using $\mathcal{C} \cong \mathcal{C}/1$) and we use $i$ to denote the isomorphism given by this Lemma. The front of the cube (\ref{diag3}) commutes trivially, since $(\mathcal{C}/Z)(T \times X, Z)$ is terminal. (See the brief discussion preceding Lemma \ref{pbesf}. The notation used in the next diagram is introduced there as well.) The bottom of the cube (\ref{diag3}) clearly commutes on object level, hence it commutes, since $(\mathcal{C}/Z^{X} \times X)(T \times X,Z^{X} \times X)$ is terminal, so in particular discrete. The remaining three faces commute by Lemma \ref{mainlem}. By gluing the diagrams (\ref{diag1}), (\ref{diag2}) and (\ref{diag3}) along their common faces and identifying terminal objects, we get the following commuting diagram.
\begin{equation} \label{diag4}
\begin{tikzcd}[row sep = large, column sep = large]
& \mathcal{C}(T, Y^{X}) \arrow[rr, "\varepsilon_{Y}*- (- \times X)"] \arrow[dd, near end, swap, "p^{X}*-"] & & \mathcal{C}(T \times X, Y) \arrow[dd, "p*-"] \\
(\mathcal{C}/Z^{X})(T,Y^{X}) \arrow[ur] \arrow[rr, crossing over, near end, "i \varepsilon_{Q}*- (- \times X)"] \arrow[ddr] & & (\mathcal{C}/Z)(T \times X, Y) \arrow[ur] & \\
& \mathcal{C}(T, Z^{X}) \arrow[rr, "(\varepsilon_{Z}*-)( - \times X)"] & & \mathcal{C}(T \times X, Z) \\
& 1 \arrow[u, "\ulcorner g \urcorner"] \arrow[urr, swap, "\ulcorner \varepsilon_{Z}(g \times 1) \urcorner"] \arrow[from=uur, crossing over] & &
\end{tikzcd}
\end{equation}
In particular, we have an induced map $\alpha$ into the pullback, $R$, of $p*-$ along $(\varepsilon_{Z}*-)( - \times X)$.
\begin{equation*}
\begin{tikzcd}[column sep = large]
\mathcal{C}(T, Y^{X}) \arrow[dr, dashed, "\alpha"] \arrow[drr, bend left=15, "(\varepsilon_{Y}*-)(- \times X)"] \arrow[ddr, bend right=15, swap, "p^{X}*-"] & & \\
& R \arrow[r, "\pi_{2}"] \arrow[d, swap, "\pi_{1}"] & \mathcal{C}(T \times X, Y) \arrow[d, "p*-"] \\
& \mathcal{C}(T, Z^{X}) \arrow[r, swap, "(\varepsilon_{Z}*-)( - \times X)"] & \mathcal{C}(T \times X, Z)
\end{tikzcd}
\end{equation*}
Recall that our goal is to show that $(\varepsilon_{Y}*-) (- \times X)$ is e.s.(f.). We claim that it suffices to show that the pullback of $\alpha$ along $\ulcorner g \urcorner$ is e.s.(f.). By assumption, $(\varepsilon_{Z}*-)( - \times X)$ is e.s.(f.). This means that $\pi_{2}$ is e.s.(f.) as well, by Lemma \ref{esfpblem}, as $p*-$ is an isofibration according to Lemma \ref{fiblem}. Lemma \ref{esflem} shows that we are done once $\alpha$ is e.s.(f.). The map $p^{X}*-$ is an isofibration by Lemma \ref{fiblem}, so the claim follows from Lemma \ref{1lem}, since $g$ is arbitrary.

The pullback of $R$ along $\ulcorner g \urcorner$ is the same as the pullback of $\mathcal{C}(T \times X, Y)$ along $\ulcorner \varepsilon_{Z}(g \times 1) \urcorner$, by pullback pasting. We see that the diagram (\ref{diag4}) induces two maps, such that the following square commutes.
\begin{equation*}
\begin{tikzcd}
(\mathcal{C}/Z^{X})(T,Y^{X}) \arrow[r, "i (\varepsilon_{Q}*-) (- \times X)"] \arrow[d, dashed] & (\mathcal{C}/Z)(T \times X,Y) \arrow[d, dashed] \\
\ulcorner g \urcorner^{*} \mathcal{C}(T,Y^{X}) \arrow[r, swap, "\ulcorner g \urcorner^{*} \alpha"] & \ulcorner \varepsilon_{Z}(g \times 1) \urcorner^{*}\mathcal{C}(T \times X,Y)
\end{tikzcd}
\end{equation*}
The two induced maps are e.s.f. by Lemma \ref{pobjlem} (using $\mathcal{C} \cong \mathcal{C}/1$). Moreover, $i \varepsilon_{Q*} (- \times X)$ is e.s.(f.), since $\varepsilon_{Q*} (- \times X)$ is e.s.(f.) by assumption. It follows that $\ulcorner g \urcorner^{*} \alpha$ is e.s.(f.) by Lemma \ref{esflem}, as needed.
\end{proof}

\begin{prop} \label{mainprop2}
Let $\mathcal{C}$ be a path category with weak (resp. strong) homotopy $\Pi$-types. Given a fibration $f : I \longrightarrow J$ in $\mathcal{C}$, a fibration $p : X \longrightarrow Y$ in $\mathcal{C}(I)$ and a weak (strong) homotopy $\Pi$-type $(\Pi_{f}Y, \varepsilon_{Y})$, there is a weak (strong) homotopy $\Pi$-type $(\Pi_{f}X, \varepsilon_{X})$ and a fibration $\Pi_{f} p : \Pi_{f}X \longrightarrow \Pi_{f}Y$ in $\mathcal{C}(J)$ such that the square
\begin{equation*}
\begin{tikzcd}
(\Pi_{f}X) \times_{J} I \arrow[r, "\varepsilon_{X}"] \arrow[d, swap, "(\Pi_{f}p) \times_{J} 1"] & X \arrow[d, "p"] \\
(\Pi_{f}Y) \times_{J} I \arrow[r, swap, "\varepsilon_{Y}"] & Y
\end{tikzcd}
\end{equation*}
commutes.
\end{prop}

\begin{proof}
Take the pullback
\begin{equation*}
\begin{tikzcd}
Q \arrow[r, "e"] \arrow[d, swap, "q"] & X \arrow[d, "p"] \\
(\Pi_{f}Y) \times_{J} I \arrow[r, swap, "\varepsilon_{Y}"] & Y
\end{tikzcd}
\end{equation*}
and define $\Pi_{f} p : \Pi_{f}X \longrightarrow \Pi_{f}Y$ as $\Pi_{\pi_{1}}q : \Pi_{\pi_{1}} Q \longrightarrow \Pi_{f}Y$, where $\pi_{1} : (\Pi_{f}Y) \times_{J} I \longrightarrow \Pi_{f}Y$. Put $\varepsilon_{X} = e \varepsilon_{Q}$. Clearly $\Pi_{f} p$ meets our requirements. We must show that $(\Pi_{f}X, \varepsilon_{X})$ is a weak (strong) homotopy $\Pi$-type. Let $g : T \longrightarrow \Pi_{f}Y$ be an arbitrary morphism in $\mathcal{C}/J$. Since the square
\begin{equation*}
\begin{tikzcd}
\mathcal{C}/\Pi_{f}Y \arrow[r, "- \times_{J} I"] \arrow[d] & \mathcal{C}/(\Pi_{f}Y) \times_{J} I \arrow[d] \\
\mathcal{C}/J \arrow[r, swap, "- \times_{J} I"] & \mathcal{C}/I
\end{tikzcd}
\end{equation*}
of homtopical functors commutes, Lemma \ref{mainlem} shows that that the cube
\begin{equation} \label{diagpi1}
\begin{tikzcd}[row sep = large, column sep = tiny]
& (\mathcal{C}/J)(T,\Pi_{f}X) \arrow[rr, "- \times_{J} I"] \arrow[dd, near end, swap, "(\Pi_{f}p)*-"] & & (\mathcal{C}/I)(T \times_{J} I, (\Pi_{f} X) \times_{J} I) \arrow[dd, "((\Pi_{f}p) \times_{J} 1)*-"] \\
(\mathcal{C}/\Pi_{f}Y)(T,\Pi_{f} X) \arrow[rr, crossing over, near end, "- \times_{J} I"] \arrow[dd, swap, "(\Pi_{f} p)*-"] \arrow[ur] & & (\mathcal{C}/(\Pi_{f} Y) \times_{J} I)(T \times_{J} I, (\Pi_{f} X) \times_{J} I) \arrow[ur] & \\
& (\mathcal{C}/J)(T,\Pi_{f} Y) \arrow[rr, near start, "- \times_{J} I"] & & (\mathcal{C}/I)(T \times_{J} I, (\Pi_{f} Y) \times_{J} I) \\
(\mathcal{C}/ \Pi_{f} Y)(T, \Pi_{f} Y) \arrow[rr, "- \times_{J} I"] \arrow[ur] & & (\mathcal{C}/(\Pi_{f} Y) \times_{J} I)(T \times_{J} I, (\Pi_{f} Y) \times_{J} I) \arrow[ur] \arrow[from=uu, crossing over, swap, near end, "((\Pi_{f} p) \times_{J} 1)*-"] &
\end{tikzcd}
\end{equation}
commutes. The prism
\begin{equation} \label{diagpi2}
\begin{tikzcd}[row sep = large, column sep = -1.2 em]
& (\mathcal{C}/I)(T \times_{J} I, (\Pi_{f} X) \times_{J} I) \arrow[rr, "\varepsilon_{Q}*-"] \arrow[ddr, near start, "((\Pi_{f} p) \times_{J} 1)*-"] & & (\mathcal{C}/I)(T \times_{J} I, Q) \arrow[ddl, dash pattern=on 43pt off 17pt, near start, "q*-"] \\
(\mathcal{C}/(\Pi_{f}Y) \times_{J} I)(T \times_{J} I, (\Pi_{f} X) \times_{J} I) \arrow[ur] \arrow[rr, crossing over, "\varepsilon_{Q}*-"] \arrow[ddr, swap, "((\Pi_{f} p) \times_{J} 1)*-"] & & (\mathcal{C}/ (\Pi_{f} Y) \times_{J} I) (T \times_{J} I, Q) \arrow[ur] \arrow[ddl, crossing over, swap, "q*-"] & \\
& & (\mathcal{C}/I)(T \times_{J} I, (\Pi_{f} Y) \times_{J} I)  & \\
& (\mathcal{C}/(\Pi_{f} Y) \times_{J} I)(T \times_{J} I, (\Pi_{f} Y) \times_{J} I) \arrow[ur] & &
\end{tikzcd}
\end{equation}
also commutes by Lemma \ref{mainlem}. Lastly, consider the cube
\begin{equation} \label{diagpi3}
\begin{tikzcd}[row sep = large, column sep = tiny]
& (\mathcal{C}/I)(T \times_{J} I,Q) \arrow[rr, "e*-"] \arrow[dd, near end, swap, "q*-"] & & (\mathcal{C}/I)(T \times_{J} I, X) \arrow[dd, "p*-"] \\
(\mathcal{C}/(\Pi_{f} Y) \times_{J} I)(T \times_{J} I,Q) \arrow[rr, crossing over, near end, "i"] \arrow[dd, swap, "q*-"] \arrow[ur] & & (\mathcal{C}/Y)(T \times_{J} I, X) \arrow[ur] & \\
& (\mathcal{C}/I)(T \times_{J} I, (\Pi_{f} Y) \times_{J} I) \arrow[rr, near start, "\varepsilon_{Y}*-"] & & (\mathcal{C}/I)(T \times_{J} I, Y) \\
(\mathcal{C}/(\Pi_{f}Y) \times_{J} I)(T \times_{J} I,(\Pi_{f} Y) \times_{J} I) \arrow[rr, "\varepsilon_{Y}*-"] \arrow[ur] & & (\mathcal{C}/Y)(T \times_{J} I, Y) \arrow[ur] \arrow[from=uu, crossing over, swap, near end, "p*-"] &
\end{tikzcd}
\end{equation}
The top of the cube (\ref{diagpi3}) commutes by Lemma and we use $i$ to denote the isomorphism given by this Lemma. The front of the cube (\ref{diagpi3}) commutes trivially, since $(\mathcal{C}/Y)(T \times_{J} I, Y)$ is terminal. The bottom of the cube (\ref{diagpi3}) clearly commutes on object level, hence it commutes, since $(\mathcal{C}/(\Pi_{f}Y \times_{J} I)(T \times_{J} I,(\Pi_{f} Y) \times_{J} I)$ is terminal, so in particular discrete. The remaining three faces commute by Lemma \ref{mainlem}. By gluing the diagrams (\ref{diagpi1}), (\ref{diagpi2}) and (\ref{diagpi3}) along their common faces and identifying terminal objects, we get the following commuting diagram.
\begin{equation} \label{diagpi4}
\begin{tikzcd}[row sep = large, column sep = large]
& (\mathcal{C}/J)(T, \Pi_{f} X) \arrow[rr, "(\varepsilon_{X}*-) (- \times_{J} I)"] \arrow[dd, near end, swap, "(\Pi_{f} p)*-"] & & (\mathcal{C}/I)(T \times_{J} I, X) \arrow[dd, "p*-"] \\
(\mathcal{C}/\Pi_{f} Y)(T,\Pi_{f} X) \arrow[ur] \arrow[rr, crossing over, near end, "i (\varepsilon_{Q}*-) (- \times_{J} I)"] \arrow[ddr] & & (\mathcal{C}/Y)(T \times_{J} I, X) \arrow[ur] & \\
& (\mathcal{C}/J)(T, \Pi_{f} Y) \arrow[rr, "(\varepsilon_{Y}*-)( - \times_{J} I)"] & & (\mathcal{C}/I)(T \times_{J} I, Y) \\
& 1 \arrow[u, "\ulcorner g \urcorner"] \arrow[urr, swap, "\ulcorner \varepsilon_{Y}(g \times_{J} 1) \urcorner"] \arrow[from=uur, crossing over] & &
\end{tikzcd}
\end{equation}
In particular, we have an induced map $\alpha$ into the pullback, $R$, of $p*-$ along $(\varepsilon_{Y}*-)( - \times_{J} I)$.
\begin{equation*}
\begin{tikzcd}[column sep = huge]
(\mathcal{C}/J)(T, \Pi_{f} X) \arrow[dr, dashed, "\alpha"] \arrow[drr, bend left=15, "(\varepsilon_{X}*-) (- \times_{J} I)"] \arrow[ddr, bend right=15, swap, "(\Pi_{f} p)*-"] & & \\
& R \arrow[r, "\pi_{2}"] \arrow[d, swap, "\pi_{1}"] & (\mathcal{C}/I)(T \times_{J} I, X) \arrow[d, "p*-"] \\
& (\mathcal{C}/J)(T, \Pi_{f} Y) \arrow[r, swap, "(\varepsilon_{Y}*-)( - \times_{J} I)"] & (\mathcal{C}/I)(T \times_{J} I, Y)
\end{tikzcd}
\end{equation*}
Recall that our goal is to show that $(\varepsilon_{X}*-) (- \times_{J} I)$ is e.s.(f.). We claim that it suffices to show that the pullback of $\alpha$ along $\ulcorner g \urcorner$ is e.s.(f.). By assumption, $(\varepsilon_{Y}*-)( - \times_{J} I)$ is e.s.(f.). This means that $\pi_{2}$ is e.s.(f.) as well, by Lemma \ref{esfpblem}, as $p*-$ is an isofibration according to Lemma \ref{fiblem}. Lemma \ref{esflem} shows that we are done once $\alpha$ is e.s.(f.). The map $(\Pi_{f}p)*-$ is an isofibration by Lemma \ref{fiblem}, so the claim follows from Lemma \ref{1lem}, since $g$ is arbitrary.

The pullback of $R$ along $\ulcorner g \urcorner$ is the same as the pullback of $(\mathcal{C}/I)(T \times_{J} I, X)$ along $\ulcorner \varepsilon_{Y}(g \times_{J} 1) \urcorner$, by pullback pasting. We see that the diagram (\ref{diagpi4}) induces two maps, such that the following square commutes.
\begin{equation*}
\begin{tikzcd}
(\mathcal{C}/\Pi_{f}Y)(T,\Pi_{f}X) \arrow[r, "i (\varepsilon_{Q}*-) (- \times_{J} I)"] \arrow[d, dashed] & (\mathcal{C}/Y)(T \times_{J} I,X) \arrow[d, dashed] \\
\ulcorner g \urcorner^{*} (\mathcal{C}/\Pi_{f}Y)(T,\Pi_{f}X) \arrow[r, swap, "\ulcorner g \urcorner^{*} \alpha"] & \ulcorner \varepsilon_{Z}(g \times_{J} 1) \urcorner^{*}(\mathcal{C}/Y)(T \times_{J} I,X)
\end{tikzcd}
\end{equation*}
The two induced maps are e.s.f. by Lemma \ref{pbesf}. Moreover, $i (\varepsilon_{Q}*-) (- \times_{J} I)$ is e.s.(f.), since $\varepsilon_{Q*} (- \times_{J} I)$ is e.s.(f.) by assumption. It follows that $\ulcorner g \urcorner^{*} \alpha$ is e.s.(f.) by Lemma \ref{esflem}, as needed.
\end{proof}

The construction of $\Pi_{f}X$ used in the proof of Proposition \ref{mainprop2} shows that we can `vertically compose' $\Pi$-types. The following lemma shows that we can `horizontally compose' $\Pi$-types as well.

\begin{lem}
Let $\mathcal{C}$ be a path category with weak (resp. strong) homotopy $\Pi$-types and let $f : X \longrightarrow I$, $g : I \longrightarrow J$ and $h : J \longrightarrow K$ be fibrations in $\mathcal{C}$. Then $\Pi_{h} \Pi_{g} X$ is a weak (resp. strong) homotopy $\Pi$-type for $f$ and $hg$.
\end{lem}

\begin{proof}
Consider an arbitrary map $T \longrightarrow K$. Then by Lemma \ref{mainlem} the diagram
\begin{equation*}
\begin{tikzcd}
(\mathcal{C}/K)(T, \Pi_{h} \Pi_{g} X) \arrow[d, swap, "h^{*}"] & \\
(\mathcal{C}/J)(h^{*}T, h^{*} \Pi_{h} \Pi_{g} X) \arrow[r, "g^{*}"] \arrow[d, swap, "\varepsilon_{\Pi_{g}X} * -"] & (\mathcal{C}/I)(g^{*} h^{*}T, g^{*} h^{*} \Pi_{h} \Pi_{g} X) \arrow[d, "g^{*}(\varepsilon_{\Pi_{g}}) * -"] \\
(\mathcal{C}/J)(h^{*}T, \Pi_{g} X) \arrow[r, swap, "g^{*}"] & (\mathcal{C}/I)(g^{*} h^{*} T, g^{*} \Pi_{g} X) \arrow[d, "\varepsilon_{X} * -"] \\
& (\mathcal{C}/I)(g^{*} h^{*} T, X)
\end{tikzcd}
\end{equation*}
commutes. Since $(\varepsilon_{\Pi_{g}X} * -) h^{*}$ and $(\varepsilon_{X} * -) g^{*}$ are e.s.(f.) by assumption, this implies that $((\varepsilon_{X} g^{*}(\varepsilon_{\Pi_{g}})) * -) (hg)^{*}$ is e.s.(f.) by Lemma \ref{esflem}.
\end{proof}

Suppose that $\mathcal{C}$ be a path category with weak homotopy exponentials and weak homotopy $\Pi$-types and let $(Y^{X}, \varepsilon_{Y})$ be a weak homotopy exponential for objects $X$ and $Y$. Then we can choose $(Y^{X} \times Y^{X}, \varepsilon_{Y \times Y} )$ as a weak homotopy exponential for $X$ and $Y \times Y$, where $\varepsilon_{Y \times Y} = (\varepsilon_{Y}(\pi_{1}, \pi_{3}), \varepsilon_{Y}(\pi_{2}, \pi_{3}))$. This follows readily from the fact that products commute with path objects. We apply Proposition \ref{mainprop1} to the weak homotopy exponential $(Y^{X} \times Y^{X}, \varepsilon_{Y \times Y} )$ and the fibration $(s,t) : PY \longrightarrow Y \times Y$. This gives us a weak homotopy exponential $((PY)^{X}, \varepsilon_{PY})$ and a fibration $(s^{X}, t^{X}) : (PY)^{X} \longrightarrow Y^{X} \times Y^{X}$ such that the square
\begin{equation*}
\begin{tikzcd}
(PY)^{X} \times X \arrow[d, swap, "{(s^{X}, t^{X}) \times 1}"] \arrow[r, "\varepsilon_{PY}"] & PY \arrow[d, "{(s,t)}"] \\
Y^{X} \times Y^{X} \times X \arrow[r, swap, "\varepsilon_{Y \times Y}"] &  Y \times Y
\end{tikzcd}
\end{equation*}
commutes. The commutative square
\begin{equation*}
\begin{tikzcd}
Y^{X} \times X \arrow[d, swap, "{\Delta \times 1}"] \arrow[r, "r \varepsilon_{Y}"] & PY \arrow[d, "{(s,t)}"] \\
Y^{X} \times Y^{X} \times X \arrow[r, swap, "\varepsilon_{Y \times Y}"] &  Y \times Y
\end{tikzcd}
\end{equation*}
induces a map $(\Delta \times 1, r \varepsilon_{Y})  : Y^{X} \times X \longrightarrow Q$, with $Q$ as in the proof of Proposition \ref{mainprop1}. Using the fact that $(PY)^{X}$ is defined as a weak homotopy $\Pi$-type for $Q$, this in turn induces a map $r^{X} : Y^{X} \longrightarrow (PY)^{X}$, over $Y^{X} \times Y^{X}$, such that $\varepsilon_{Q} (r^{X} \times 1) \simeq_{Y^{X} \times Y^{X} \times X} (\Delta \times 1, r \varepsilon_{Y})$, which is equivalent to $\varepsilon_{PY} (r^{X} \times 1) \simeq_{Y \times Y} r \varepsilon_{Y}$. Lastly, let $\phi$ be a filler of the square
\begin{equation*}
\begin{tikzcd}
Y^{X} \arrow[d, swap, "r"] \arrow[r, "r^{X}"] & (PY)^{X} \arrow[d, "{(s^{X},t^{X})}"] \\
P(Y^{X}) \arrow[r, swap, "{(s,t)}"] &  Y^{X} \times Y^{X}
\end{tikzcd}
\end{equation*}

Translated to type theory, the map $\phi$ witnesses that if two functions $f,g :X \longrightarrow Y$ are equal, then $f(x)$ and $g(x)$ are equal on all inputs $x$. The converse of this statement -- if $f(x)$ and $g(x)$ are equal on all inputs $x$, then $f$ and $g$ are equal -- is called function extensionality. This property corresponds to the map $\phi$ being a weak equivalence. The following proposition shows that strong homotopy exponentials are exactly those weak homotopy exponentials satisfying function extensionality.

\begin{prop} \label{explem}
In the situation as described above, $Y^{X}$ is a strong homotopy exponential if and only if $\phi : P(Y^{X}) \longrightarrow (PY)^{X}$ is a weak equivalence.
\end{prop}

\begin{proof}
We claim that $\varepsilon_{PY} ( \phi \times 1) \simeq_{Y \times Y} P \varepsilon_{Y} ( 1 \times r)$. It suffices to prove that both maps fill the square
\begin{equation*}
\begin{tikzcd}[column sep=large]
Y^{X} \times X \arrow[d, swap, "r \times 1"] \arrow[r, "r \varepsilon_{Y}"] & PY \arrow[d, "{(s,t)}"] \\
P(Y^{X}) \times X \arrow[r, swap, "{\varepsilon_{Y \times Y}(  (s,t) \times 1)}"] &  Y \times Y
\end{tikzcd}
\end{equation*}
We calculate
\begin{align*}
(s,t) \varepsilon_{PY} ( \phi \times 1) & = \varepsilon_{Y \times Y}((s^{X}, t^{X}) \times 1)( \phi \times 1) \\
& = \varepsilon_{Y \times Y} (((s^{X}, t^{X}) \phi ) \times 1) \\
& = \varepsilon_{Y \times Y} ( (s, t) \times 1)
\end{align*}
and
\begin{align*}
\varepsilon_{PY} ( \phi \times 1 )( r \times 1) & = \varepsilon_{PY} (( \phi r) \times 1) \\
& \simeq_{Y \times Y} \varepsilon_{PY} ( r^{X} \times 1) \\
& \simeq _{Y \times Y} r \varepsilon_{Y},
\end{align*}
using Corollary \ref{cor1} and the fact that $r^{X} \times 1 \simeq_{Y^{X} \times Y^{X} \times X} (\phi r) \times 1$, since $r^{X} \simeq_{Y^{X} \times Y^{X}} \phi r$. We also see
\begin{align*}
(s,t) P \varepsilon_{Y} ( 1 \times r) & = ( \varepsilon_{Y}(s \pi_{1}, s \pi_{2}), \varepsilon_{Y}(t \pi_{1}, t \pi_{2}) ) (1 \times r) \\
& = ( \varepsilon_{Y}(s \pi_{1}, s r \pi_{2}), \varepsilon_{Y}(t \pi_{1}, t r \pi_{2}) ) \\
& = ( \varepsilon_{Y}(s \pi_{1}, \pi_{2}), \varepsilon_{Y}(t \pi_{1}, \pi_{2}) ) \\
& =  \varepsilon_{Y \times Y}( (s,t) \times 1)
\end{align*}
and
\begin{align*}
P \varepsilon_{Y} (1 \times r) (r \times 1) & = P \varepsilon_{Y} (r \times r) \\
& \simeq_{Y \times Y} r \varepsilon_{Y}.
\end{align*}

Suppose that $Y^{X}$ is strong. Then since $\varepsilon_{PY} : \varepsilon_{Y} (s^{X} \times 1)\simeq \varepsilon_{Y} (t^{X} \times 1)$, the fullness property of $Y^{X}$ produces a homotopy $\psi : s^{X} \simeq t^{X}$ such that $P \varepsilon_{Y} (\psi \times r) \simeq_{Y \times Y} \varepsilon_{PY}$. Applying the fullness property of $Y^{X}$ once more, but now to the homotopy
\begin{align*}
\varepsilon_{PY}( (\phi \psi) \times 1) & = \varepsilon_{PY} (\phi
\times 1)(\psi \times 1) \\
& \simeq_{Y \times Y} P \varepsilon_{Y} (1 \times r) ( \psi \times 1) \\
& = P \varepsilon_{Y} ( \psi \times r) \\
& \simeq_{Y \times Y} \varepsilon_{PY} \\
& = \varepsilon_{PY} (1 \times 1).
\end{align*}
shows that $\phi \psi \simeq 1$. This means that the diagram
\begin{equation*}
\begin{tikzcd}
(PY)^{X} \arrow[r, "\psi"] \arrow[d, swap, "s^{X}"] & P(Y^{X}) \arrow[r, "\phi"] \arrow[d, "s"] & (PY)^{X} \arrow[d, "s^{X}"] \\
PY \arrow[r, swap, "1"] & PY \arrow[r, swap, "1"] & PY 
\end{tikzcd}
\end{equation*}
exhibits $s^{X}$ as a retract of $s$, in $\mathsf{Ho}(\mathcal{C})$. Since $s$ is an isomorphism in $\mathsf{Ho}(\mathcal{C})$, so is $s^{X}$, hence $s^{X}$ is a weak equivalence in $\mathcal{C}$. It follows that $\phi$ is a weak equivalence as well, since $s^{X} \phi = s$.

Conversely, suppose that $\phi$ is a weak equivalence. Let $f, g : T \longrightarrow Y^{X}$ be maps such that $H : \varepsilon_{Y}(f \times 1) \simeq_{Y \times Y} \varepsilon_{Y}(g \times 1)$. We must find some $K : f \simeq_{Y^{X} \times Y^{X}} g$ satisfying $P \varepsilon_{Y}(K \times r) \simeq_{Y \times Y} H$. The commutative square
\begin{equation*}
\begin{tikzcd}
Y^{X} \times X \arrow[d, swap, "{(f,g) \times 1}"] \arrow[r, "H"] & PY \arrow[d, "{(s,t)}"] \\
Y^{X} \times Y^{X} \times X \arrow[r, swap, "\varepsilon_{Y \times Y}"] &  Y \times Y
\end{tikzcd}
\end{equation*}
induces a map $((f,g) \times 1, H)  : Y^{X} \times X \longrightarrow Q$. Using the fact that $(PY)^{X}$ is defined as a weak homotopy $\Pi$-type for $Q$, this in turn induces a map $L : Y^{X} \longrightarrow (PY)^{X}$, over $Y^{X} \times Y^{X}$, such that $\varepsilon_{Q} (L \times 1) \simeq_{Y^{X} \times Y^{X} \times X} ((f,g) \times 1, H)$, which is equivalent to $\varepsilon_{PY} (L \times 1) \simeq_{Y \times Y} H$. Let $\psi$ be a homotopy inverse of $\phi$, over $Y \times Y$, and put $K = \psi L$. Then
\begin{align*}
(s,t)K & = (s,t) \psi L \\
& = (s^{X}, t^{X}) L \\
& = (f,g)
\end{align*}
and
\begin{align*}
P \varepsilon_{Y}(K \times r) & = P \varepsilon_{Y}(1 \times r)(K \times 1) \\
& \simeq_{Y \times Y} \varepsilon_{PY}(\phi \times 1)(K \times 1) \\
& = \varepsilon_{PY}( (\phi \psi) \times 1)(L \times 1) \\
& \simeq_{Y \times Y} \varepsilon_{PY}(L \times 1) \\
& \simeq_{Y \times Y} H,
\end{align*}
as desired.
\end{proof}

\begin{cor} \label{verystr}
Let $\mathcal{C}$ be a path category with strong homotopy $\Pi$-types and let $(Y^{X}, \varepsilon_{Y})$  be a strong homotopy exponential for $X$ and $Y$. Then the functor
\begin{equation*}
\begin{tikzcd}
\mathcal{C}(T,Y^{X}) \arrow[r, "- \times X"] & \mathcal{C}(T \times X, Y^{X} \times X) \arrow[r, "\varepsilon_{Y}*-"] & \mathcal{C}(T \times X, Y)
\end{tikzcd}
\end{equation*}
is an equivalence of categories, for every object $T$.
\end{cor}

\begin{proof}
We must prove that $(\varepsilon_{Y}*-)(- \times X)$ is faithful. So take $H, K : T \longrightarrow P(Y^{X})$ and assume that $P \varepsilon_{Y} (H \times r) \simeq_{Y \times Y} P \varepsilon_{Y} (K \times r)$. We wish to derive that $H \simeq_{Y^{X} \times Y^{X}} K$. Construct $(PY)^{X}$ as in Proposition \ref{explem}. Then according to this same Proposition, $\phi : P(Y^{X}) \longrightarrow (PY)^{X}$ is a weak equivalence, which implies that $r^{X} : Y^{X} \longrightarrow (PY)^{X}$ is a weak equivalence. This shows that $(PY)^{X}$ is a path object of $Y^{X}$ and we may therefore assume that $P(Y^{X}) = (PY)^{X}$.  In the proof of Proposition \ref{explem} we have seen that $\varepsilon_{PY} (\phi \times 1) \simeq_{Y \times Y} P \varepsilon_{Y} (1 \times r)$, which reduces to $\varepsilon_{PY} \simeq_{Y \times Y} P \varepsilon_{Y} (1 \times r)$ for us, since the assumption $P(Y^{X}) = (PY)^{X}$ allows us to choose $\phi = 1$. This implies that $\varepsilon_{PY} (H \times 1) \simeq_{Y \times Y} \varepsilon_{PY} (K \times 1)$ and hence $\varepsilon_{Q} (H \times 1) \simeq_{Y \times Y} \varepsilon_{Q} (K \times 1)$. We can now use the fullness property of the strong homotopy $\Pi$-type $\Pi_{\pi_{1}}Q = (PY)^{X}$ to conclude that $H \simeq_{Y^{X} \times Y^{X}} K$.
\end{proof}

\appendix

\section{Two Lemmas on pullbacks} \label{sec6}

In this section we prove two somewhat technical Lemmas on pullbacks. Lemma \ref{pobjlem} is solely used in the proof of Lemma \ref{pbesf}.

\begin{lem} \label{pobjlem}
Let $f : X \longrightarrow Y$ be a fibration in a path category $\mathcal{C}$. Take pullbacks
\begin{equation*}
\begin{tikzcd}
Y \times_{Y \times Y} PX \arrow[r] \arrow[d] & PX \arrow[d, "{(fs,ft)}"] \\
Y \arrow[r, swap, "\Delta"] & Y \times Y
\end{tikzcd}
\end{equation*}
and
\begin{equation*}
\begin{tikzcd}[column sep=huge]
Q \arrow[r] \arrow[d, swap, "{(p_{1}, p_{2})}"] & P_{Y \times Y}PY \arrow[d, "{(s,t)}"] \\
Y \times_{Y \times Y} PX \arrow[r, swap, "{(r \pi_{1}, (Pf) \pi_{2})}"] & PY \times_{Y \times Y} PY
\end{tikzcd}
\end{equation*}
Then $Q$ is a path object of $X$, over $Y$.
\end{lem}

\begin{proof}
By Theorem 2.28 of \cite{Path}, we may assume that $Pf : PX \longrightarrow PY$ is a fibration and that $(Pf)r = rf$. The latter assumption ensures that we have a map $r_{Q} = ((f,r),rrf) : X \longrightarrow Q$.

We claim that $r_{Q}$ is a weak equivalence. Note that $p_{2} r_{Q} = r$, so it suffices to prove that $p_{2}$ is a weak equivalence. We can almost argue that this follows from the fact that $p_{2}$ is the homotopy pullback of $r$ along $Pf$ in $\mathcal{C}(Y \times Y)$, however, $\Delta : Y \longrightarrow Y \times Y$ is not guaranteed to lie in $\mathcal{C}(Y \times Y)$. So instead, we ask the reader to verify that $Q$ fits in the following diagram, composed entirely of pullback squares
\begin{equation*}
\begin{tikzcd}[column sep=huge]
Q \arrow[r] \arrow[rr, bend left=15, "p_{2}"] \arrow[d] & \cdot \arrow[r] \arrow[d] & PX \arrow[d, "Pf"] \\
\cdot \arrow[r] \arrow[d]  & P_{Y \times Y} PY \arrow[r, swap, "t"] \arrow[d , "s"]  & PY \\
Y \arrow[r, swap, "r"] & PY & {}
\end{tikzcd}
\end{equation*}
Since weak equivalences are stable under pullbacks along fibrations, we see that $p_{2}$ is a composite of two weak equivalences.

We also ask the reader to verify that $Y \times_{Y \times Y} PX$ fits in the following pullback square
\begin{equation*}
\begin{tikzcd}
Y \times_{Y \times Y} PX \arrow[r] \arrow[d, swap, "{(s \pi_{2}, t \pi_{2})}"] & PX \arrow[d, "{(s,t)}"] \\
X \times_{Y} X \arrow[r, swap, "{(\pi_{1}, \pi_{2})}"] & X \times X
\end{tikzcd}
\end{equation*}
Since fibrations are stable under pullback, we see that $(s_{Q}, t_{Q}) = (s \pi_{2}, t \pi_{2})(p_{1}, p_{2}) : Q \longrightarrow X \times_{Y} X$ is a composite of two fibrations. One easily checks that $(s_{Q}, t_{Q})r_{Q} = (1,1)$ and we conclude that $Q$ is a path object of $X$, over $Y$.
\end{proof}

Let $\mathcal{C}$ be a path category and fix a fibration $g : Y \longrightarrow Z$ in $\mathcal{C}$. This gives an evident fibrant object preserving homotopical functor $g_{*} : \mathcal{C} / Y \longrightarrow \mathcal{C} / Z$. Moreover, let $f : X \longrightarrow Y$ be a fibration and let $h : W \longrightarrow Y$ be a morphism. Then according to Lemma \ref{mainlem} the square
\begin{equation*}
\begin{tikzcd}
(\mathcal{C}/Y)(W,X) \arrow[r, "g_{*}"] \arrow[d, swap, "f*-"] & (\mathcal{C}/Z)(W,X) \arrow[d, "f*-"] \\
(\mathcal{C}/Y)(W,Y) \arrow[r, swap, "g_{*}"] & (\mathcal{C}/Z)(W,Y)
\end{tikzcd}
\end{equation*}
commutes. Note that $(\mathcal{C}/Y)(W,Y) \cong 1$, since (the identity on) $Y$ is the terminal object in $\mathcal{C}(Y)$ and we may choose $P1 = 1$ in any path category. We denote $g_{*} : 1 \longrightarrow (\mathcal{C}/Z)(W,Y)$ by $\ulcorner h \urcorner$, since it sends the unique object of $1$ to the object $h$ of $(\mathcal{C}/Z)(W,Y)$. We are interested in the induced map $J : (\mathcal{C}/Y)(W,X) \longrightarrow  \ulcorner h \urcorner^{*} (\mathcal{C}/Z)(W,X)$ into the pullback of $(\mathcal{C}/Z)(W,X)$ along $\ulcorner h \urcorner$.

\begin{lem} \label{pbesf}
The functor $J : (\mathcal{C}/Y)(W,X) \longrightarrow  \ulcorner h \urcorner^{*} (\mathcal{C}/Z)(W,X)$, as described in the preceding paragraph, is e.s.f.
\end{lem}

\begin{proof}
We can give an explicit description of $\ulcorner h \urcorner^{*} (\mathcal{C}/Z)(W,X)$. It is the subgroupoid of $(\mathcal{C}/Z)(W,X)$, whose objects are those $k : W \longrightarrow X$ satisfying $fk=h$. Its arrows are (equivalence classes of) those $H : W \longrightarrow P_{Z}X$ satisfying $PfH \simeq_{Y \times_{Z} Y} rh$. Clearly, $J$ is bijective on objects.

It remains to show that $J$ is full. Let $H : W \longrightarrow P_{Z}X$ be (a representative of) an arrow in $\ulcorner h \urcorner^{*} (\mathcal{C}/Z)(W,X)$. By our description, there exists a homotopy $\mathcal{H} : PfH \simeq_{Y \times_{Z} Y} rh$. Consider the pullback
\begin{equation*}
\begin{tikzcd}[column sep=huge]
Q \arrow[r] \arrow[d, swap, "{(p_{1}, p_{2})}"] & P_{Y \times_{Z} Y}P_{Z}Y \arrow[d, "{(s,t)}"] \\
Y \times_{Y \times_{Z} Y} P_{Z}X \arrow[r, swap, "{(r \pi_{1}, (Pf) \pi_{2})}"] & P_{Z}Y \times_{Y \times_{Z} Y} P_{Z}Y
\end{tikzcd}
\end{equation*}
as in Lemma \ref{pobjlem}, but applied to the path category $\mathcal{C}(Z)$. Then we have a map $((h,H), \mathcal{H}) : W \longrightarrow Q$. By Lemma \ref{pobjlem}, $Q$ is a path object for $X$, over $Y$, in $\mathcal{C}(Z)$. Then of course $Q$ is a path object for $X$, over $Y$, in $\mathcal{C}$ as well. Thus $((h,H), \mathcal{H})$ is (a representative of) an arrow in $(\mathcal{C}/Y)(W,X)$. Now note that $p_{2}$ is a (strict) filler for the square
\begin{equation*}
\begin{tikzcd}
X \arrow[r, "r"] \arrow[d, swap, "r_{Q}"] & P_{Z}X \arrow[d, "{(s,t)}"] \\
Q \arrow[r, swap, "{(s_{Q},t_{Q})}"] & X \times_{Z} X
\end{tikzcd}
\end{equation*}
By definition, this means that $J$ maps (the equivalence class of) $((h,H), \mathcal{H})$ to (the equivalence class of) $p_{2} ((h,H), \mathcal{H}) = H$, as needed.
\end{proof}

\begin{lem} \label{pblem}
Let $l : Z \longrightarrow I$ be a fibration in a path category $\mathcal{C}$ and let
\begin{equation*}
\begin{tikzcd}
W \arrow[r, "f"] \arrow[d, swap, "h"] & X \arrow[d, "k"] \\
Y \arrow[r, swap, "g"] & Z
\end{tikzcd}
\end{equation*}
be a pullback square in $\mathcal{C}(I)$, with $h$ and $k$ fibrations. Then for any arrow $V \longrightarrow Y$ in $\mathcal{C}$, the bijection of hom-sets $(\mathcal{C}/Y)(V, W) \cong (\mathcal{C}/Z)(V, X)$, given by the (relative) adjunction $g_{*} \dashv g^{*}$, extends to an isomorpism of groupoids. Moreover, the following square (of groupoids) commutes.
\begin{equation*}
\begin{tikzcd}
(\mathcal{C}/Y)(V,W) \arrow[r, "\cong"] \arrow[d, swap, "(lg)_{*}"] & (\mathcal{C}/Z)(V,X) \arrow[d, "l_{*}"] \\
(\mathcal{C}/I)(V,W) \arrow[r, swap, "f*-"] & (\mathcal{C}/I)(V,X)
\end{tikzcd}
\end{equation*}
\end{lem}

\begin{proof}
Recall that we may choose $P_{Y}W = Y \times_{Z} P_{Z} X$. The action on arrows in the groupoid $(\mathcal{C}/Y)(V,W)$ is defined by sending (the equivalence class of) the homotopy $H : V \longrightarrow Y \times_{Z} P_{Z} X$ to (the equivalence class of) its transpose $\pi_{2} H : V \longrightarrow P_{Z} X$. It is straightforward to verify that this defines an isomorpism of groupoids, using the naturality of the bijection $(\mathcal{C}/Y)(V, W) \cong (\mathcal{C}/Z)(V, X)$ and the fact that $g^{*} : \mathcal{C}(Z) \longrightarrow \mathcal{C}(Y)$ is an exact functor.

To show the commutativity of the square, it suffices to prove that $(Pf) \lambda_{(lg)_{*}} \simeq_{X \times_{I} X} \lambda_{l_{*}} \pi_{2}$ by verifying that both fill the square
\begin{equation*}
\begin{tikzcd}
W \arrow[r, "rf"] \arrow[d, swap, "r"] & P_{I}X \arrow[d, "{(s,t)}"] \\
P_{Y}W \arrow[r, swap, "{(fs,ft)}"] & X \times_{I} X
\end{tikzcd}
\end{equation*}
We calculate
\begin{align*}
(s,t) (Pf) \lambda_{(lg)_{*}} & = (fs, ft) \lambda_{(lg)_{*}} \\
& = (f \pi_{1}, f \pi_{2}) (s,t) \lambda_{(lg)_{*}} \\
& = (f \pi_{1}, f \pi_{2}) (s,t) \\
& = (fs, ft)
\end{align*}
and
\begin{align*}
(Pf) \lambda_{(lg)_{*}} r & \simeq_{X \times_{I} X} (Pf)r \\
& \simeq_{X \times_{I} X} rf,
\end{align*}
using that $\lambda_{(lg)_{*}} r \simeq_{W \times_{I} W} r$, while $(Pf)$ converts $\simeq_{W \times_{I} W}$ into $\simeq_{X \times_{I} X}$. Furthermore, 
\begin{align*}
(s,t) \lambda_{(l)_{*}} \pi_{2} & = (s,t) \pi_{2} \\
& = (f \pi_{1}, f \pi_{2}) (s,t) \\
& = (fs, ft)
\end{align*}
and
\begin{align*}
\lambda_{(l)_{*}} \pi_{2} r & = \lambda_{(l)_{*}} r f \\
& \simeq_{X \times_{I} X} rf.
\end{align*}
\end{proof}

\bibliographystyle{alpha}
\bibliography{PiBib2}

\end{document}